\documentclass[moor,sglanonrev]{informs4}
\usepackage{eqndefns-left} 
\RequirePackage{tgtermes}
\RequirePackage{newtxtext}
\RequirePackage{newtxmath}
\RequirePackage{bm}
\RequirePackage{endnotes}

\OneAndAHalfSpacedXII 

\usepackage{anyfontsize}
\usepackage{hyperref}
\hypersetup{
    colorlinks=true,
    linkcolor=blue}
\usepackage{tabularx}
\usepackage{multirow}
\usepackage{url}

\newcommand{\pointtoproof}[1]{$\longrightarrow$ See page \pageref{#1} for the proof.\\}
\newcommand{\mydot}[2]{\langle#1, #2\rangle}
\usepackage{thm-restate}

\usepackage{thmtools}
\usepackage{thm-restate}

\usepackage[capitalise]{cleveref}

\usepackage[sort&compress]{natbib}
 \bibpunct[, ]{[}{]}{,}{n}{}{,}%
 \def\bibfont{\small}%
\EquationsNumberedBySection 

\TheoremsNumberedBySection  
\ECRepeatTheorems  %

\MANUSCRIPTNO{MOOR-0001-2024.00}

\begin{document}

\RUNAUTHOR{Vielma}

\RUNTITLE{Convex analysis for composite functions without {K}-convexity}

\TITLE{Convex analysis for composite functions without {K}-convexity}

\ARTICLEAUTHORS{%
\AUTHOR{Juan Pablo Vielma}
\AFF{
Google Research,
 \EMAIL{jvielma@google.com}
}
} 

\ABSTRACT{%
   Composite functions have been studied for over 40 years and appear in a wide range of optimization problems. Convex analysis of these functions focuses on (i) conditions for convexity of the function based on properties of its components, (ii) formulas for the convex conjugate of the function based on those of its components and (iii) chain-rule-like formulas for the sub-differential of the function. To the best of our knowledge all existing results on this matter are based on the notion of K-convexity of functions where K is a closed convex cone. These notions can be considered elementary when K is the non-negative orthant, but otherwise may limit the accessibility of the associated results. In this work we show how standard results on perturbation function based duality can be used to recover and extend existing results without the need for K-convexity. We also provide a detailed comparison with K-convexity based results and show that, while K-convexity is not needed for the convex analysis of composite functions, it certainly aids in verifying the conditions associated with our proposed results. Finally, we provide  necessary conditions for the conjugate and sub-differential formulas that are less restrictive than those required by the existing results.
}%


\maketitle

\section{Introduction}

For $g:\mathbb{R}^m\to \mathbb{R}$ and $F:\mathbb{R}^n \to \mathbb{R}^m$ consider the composite function $g\circ F:\mathbb{R}^n\to \mathbb{R}$ given by 
\[g\circ F(x)=g(F(x)) \quad \forall x\in \mathbb{R}^n.\]

As noted in \cite{burke2021study}, functions of the form $g\circ F$ have been studied for over 40 years and appear as objective functions for a wide range of optimization problems. In particular, the case where $g\circ F$ is convex has been extensively studied (e.g. \cite{burke2021study,boct2007new,boct2008new,ioan2009generalized,combari1994sous,gissler2023note,pennanen1999graph}).

Some natural questions in the convex case include (i) conditions for convexity of $g \circ F$ based on properties of $g$ and $F$, (ii) formulas for the convex conjugate of $g\circ F$ based on the conjugates of $g$ and $F$, and (iii) chain-rule-like formulas for the sub-differential of $g\circ F$.  A standard sufficient condition for convexity of $g\circ F$ when $F$ is a real-valued function (i.e. $m=1$) is monotonicity of $g$ and convexity of $F$. A first extension of these conditions to vector-valued $F$ (i.e. $m>1$) is to consider component-wise convexity of $F$ and component-wise monotonicity of $g$. As shown in \cite{hiriart2006note}, in addition to convexity of  $g\circ F$, these component-wise conditions also yield a relatively simple formula for the convex conjugate of $g\circ F$. 

The notions of component-wise convexity of $F$ and component-wise monotonicity are based on the partial order in $\mathbb{R}^m$ induced by the closed convex cone $\mathbb{R}_+^m$ (i.e. for $x,y\in \mathbb{R}^m$ we have $x\leq y$ if and only if $y-x\in \mathbb{R}_+^m$). Hence, it is natural to extend these concepts by considering an arbitrary closed convex cone $K$ to introduce the notions of $K$-convexity and $K$-monotonicity (e.g. see Section~\ref{connection:section}). To the best of our knowledge, all existing studies of $g\circ F$ for the convex case with $m>1$ (including \cite{burke2021study,boct2007new,boct2008new,ioan2009generalized,combari1994sous,gissler2023note,pennanen1999graph}) are based on these notions of $K$-convexity and $K$-monotonicity. In particular, combining these notions with perturbation-function based dualization (e.g. \cite[Section 11.H]{rockafellar2009variational}) can result in fairly general answers to all three questions above even in the infinite-dimensional setting (e.g. \cite{bot2009conjugate,ioan2009generalized}). 

It can be argued that the need for beyond-elemental convex analysis concepts such as perturbation-functions and $K$-convexity can limit the accessibility of the results. Motivated by this potential accessibility issue, \cite{burke2021study} focused on the finite dimensional case without directly applying the perturbation-function machinery to improve the accessibility of the associated result development and proofs. However, the statements of the associated results in \cite{burke2021study} (e.g. necessary conditions and formulas) still require the notions of $K$-convexity and $K$-monotonicity. For this reason, we take a different approach and focus on avoiding the notions of $K$-convexity and $K$-monotonicity altogether to obtain results whose statements only require elemental convex analysis concepts. That is, our focus is primarily to improve the accessibility for readers that seek to apply the results without necessarily looking at the proofs. Similarly to \cite{bot2009conjugate,ioan2009generalized}, the proofs of our results are based on perturbation-function based dualization, but in contrast to \cite{bot2009conjugate,ioan2009generalized} we only use a version that does not require the notions of $K$-convexity or $K$-monotonicity. Hence, we also get some accessibility improvements on the proofs themselves.

One additional advantage of avoiding $K$-convexity is that we can focus on additional properties of the functions $g$ and $F$ (e.g polyhedrality) which have been exploited by standard perturbation-function based dualization. This allows us to obtain necessary conditions for the conjugate and sub-differential formulas that are less restrictive than those required by the existing results in \cite{burke2021study,boct2007new,boct2008new,ioan2009generalized,combari1994sous,gissler2023note,pennanen1999graph}. However, a potential disadvantage is that some of these necessary conditions may be harder to verify than those based on $K$-convexity. For instance, one of these conditions requires convexity of a perturbation function that, while simple to state, could be hard to verify. Fortunately, this joint convexity is implied by some standard $K$-convexity conditions. For this reason we also provide a detailed comparison between our results and those based on $K$-convexity. In particular, we show how the results based on $K$-convexity can be more restrictive, but can aid in verifying the less restrictive conditions we consider.

The remainder of this paper is structured as follows. In Section~\ref{notation:section} we introduce some notation and assumptions. Then, in Section~\ref{perturbationsection} we review the perturbation-function based dualization results we will use for our main results, which we present in Section~\ref{convex:comp:section}. After that, we compare our main results with those based on $K$-convexity in Section~\ref{connection:section}. Finally, some omitted proofs are presented in Section~\ref{omitted:proofs:section} and Appendix~\ref{perturbationsection:proofs}.

\subsection{Notation and assumptions}\label{notation:section}

To consider extended real valued functions we let $\overline{\mathbb{R}}=\mathbb{R}\cup\{-\infty,+\infty\}$. For simplicity we use $\mathbb{R}^p$ for some $p\geq 1$ with inner product denoted by $\mydot{\cdot}{\cdot}$ as our vector spaces, but all our results also hold for Euclidean spaces through standard vectorization transformations (e.g. the space of $n\times n$ symmetric matrices can be \emph{unrolled} to $\mathbb{R}^p$  with $p=n(n+1)/2$). For a set $S\subseteq \mathbb{R}^p$ we let $\operatorname{rint}(S)$ be its relative interior and $\operatorname{cl}(S)$ be its closure. For a function  $h:\mathbb{R}^p\to\overline{\mathbb{R}}$ we let $\operatorname{dom}(h):=\{x\in \mathbb{R}^p\,:\, h(x)<+\infty\}$ be the effective domain of $h$,  $\operatorname{epi}(h):=\{(x,\alpha)\in \mathbb{R}^p\times \mathbb{R}\,:\, h(x)\leq \alpha\}$ be the epigraph of $h$, $\operatorname{cl}(h):\mathbb{R}^p\to\overline{\mathbb{R}}$ be the function such that $\operatorname{epi}\left(\operatorname{cl}(h)\right)=\operatorname{cl}\left(\operatorname{epi}(h)\right)$, and $h^*:\mathbb{R}^p\to\overline{\mathbb{R}}$ be the convex conjugate of $h$ given by
\[h^*(v):=\sup_{x\in \mathbb{R}^p}\left\{\mydot{v}{x}-h(x)\right\}\quad \forall v\in \mathbb{R}^p.\]
We say a function $h:\mathbb{R}^p\to\overline{\mathbb{R}}$ is \emph{convex} if $\operatorname{epi}(h)$ is convex, is \emph{closed} if $\operatorname{epi}(h)$ is closed, and is  \emph{proper} if $h\not \equiv +\infty$ and $h(x)>-\infty$ for all $x\in \mathbb{R}^p$. We also let 
\begin{align*}
    \Gamma\left(\mathbb{R}^p\right)&:=\left\{h:\mathbb{R}^p\to\overline{\mathbb{R}}\,:\, \text{$h$ is proper and convex}\right\}\\
    \Gamma_0\left(\mathbb{R}^p\right)&:=\left\{h\in  \Gamma\left(\mathbb{R}^p\right)\,:\, \text{$h$ is closed}\right\}.
\end{align*}

For a closed convex cone $K\subseteq \mathbb{R}^p$ we let $K^\circ:=\left\{v\in \mathbb{R}^p\,:\, \mydot{x}{v}\leq 0 \quad \forall x\in K\right\}$.

Finally, we use the following definitions to simplify some of our statements (e.g. to avoid pathological cases of $\arg \min$ without having to restrict its definition to proper functions).
\begin{definition}\label{sub:diff:def}
For $h:\mathbb{R}^n\to\overline{\mathbb{R}}$ we let 
\begin{align*}
    \arg \min_{x\in \mathbb{R}^n} h(x)&:=
    \left\{\bar{x}\in \mathbb{R}^n\,:\, h(\bar{x})\in \mathbb{R}, \quad h(\bar{x})=\inf_{x\in \mathbb{R}^n}h(x) \right\}\\
    \arg \max_{x\in \mathbb{R}^n} h(x)&:=\left\{\bar{x}\in \mathbb{R}^n\,:\, h(\bar{x})\in \mathbb{R}, \quad h(\bar{x})=\sup_{x\in \mathbb{R}^n}h(x) \right\}.
\end{align*}

If $h$ is convex and $\bar{x}\in \mathbb{R}^n$ we let 
\[\partial h(\bar{x}):=\begin{cases}\left\{v\in \mathbb{R}^n\,:\, h(x) \geq h(\bar{x})+\mydot{v}{\left(x-\bar{x}\right)}\quad \forall x\in\mathbb{R}^n\right\}&h(\bar{x})\in \mathbb{R}\\
\emptyset &\text{o.w.}\end{cases}.\]
\end{definition}
\section{Perturbation based Lagrangian duality}\label{perturbationsection}

For  $f\in \Gamma_0\left(\mathbb{R}^n\times \mathbb{R}^m\right)$ \cite[Section 11.H]{rockafellar2009variational} presents properties of the primal-dual pair 
\begin{subequations}\label{baseprimaldualpair}
\begin{align}
    \operatorname{opt}_P&:=\inf_{x\in \mathbb{R}^n} f(x,0)\\
    \operatorname{opt}_D&:=\sup_{y\in \mathbb{R}^m} -f^*(0,y).
\end{align}
\end{subequations}
For instance, \cite[Theorem 11.39]{rockafellar2009variational} shows that $\operatorname{opt}_P\geq \operatorname{opt}_D$ and gives conditions for this to hold at equality.

However, the results in \cite[Section 11.H]{rockafellar2009variational}  also imply properties of the slightly more general primal-dual pairs
\begin{subequations}\label{generalprimaldualpair}
\begin{align}
    \operatorname{opt}_P(\bar{u}, \bar{v})&:=\inf_{x\in \mathbb{R}^n} f(x,\bar{u})-\mydot{\bar{v}}{x}\\
    \operatorname{opt}_D(\bar{u}, \bar{v})&:=\sup_{y\in \mathbb{R}^m} \mydot{\bar{u}}{y} -f^*(\bar{v},y)
\end{align}
\end{subequations}
 whose parametrization on $(\bar{v},\bar{u})\in \mathbb{R}^n\times \mathbb{R}^m$ can be used to study the conjugate functions (e.g. $ \operatorname{opt}_P(0, \bar{v})=-(f(\cdot,0))^*(\bar{v})$). We now summarize these properties and connect them to the generalized nonlinear programming framework of \cite{rockafellar2023augmented}. While all these results follow from \cite[Section 11.H]{rockafellar2009variational} we give detailed proofs in Appendix~\ref{perturbationsection:proofs} for completeness. We begin with the basic definitions of Lagrangian function, primal and dual value functions, and primal and dual optimal sets. 

\begin{definition}
For $f:\mathbb{R}^n\times \mathbb{R}^m\to \overline{\mathbb{R}}$, let $l:\mathbb{R}^n\times \mathbb{R}^m\to \overline{\mathbb{R}}$ be the \emph{Lagrangian function} given by 
\begin{equation}\label{ldef}
    l(x,y):= -\left(f(x,\cdot)\right)^*(y).
\end{equation}
In addition, for $\bar{v}\in \mathbb{R}^n$, and $\bar{u}\in \mathbb{R}^m$, let $p_{\bar{v}}:\mathbb{R}^m\to \overline{\mathbb{R}}$ be the \emph{primal value function} and $q_{\bar{u}}:\mathbb{R}^n\to\overline{\mathbb{R}}$ be the \emph{negative of the dual value function} given by 
\begin{subequations}\label{pqdef}
\begin{alignat}{3}
    p_{\bar{v}}(\bar{u})&:=\inf_{x\in \mathbb{R}^n}\left\{ f(x,\bar{u})-\mydot{\bar{v}}{x}\right\}&&=-\sup_{x\in \mathbb{R}^n} \left\{\mydot{\bar{v}}{x} -f(x,\bar{u})\right\}\\
    q_{\bar{u}}(\bar{v})&:=\inf_{y\in\mathbb{R}^m}\left\{ f^*(\bar{v},y) - \mydot{\bar{u}}{y}\right\}&&=-\sup_{y\in \mathbb{R}^m}\left\{\mydot{\bar{u}}{y}-f^*(\bar{v},y)\right\}.\label{qdef}
    \end{alignat}
    \end{subequations}
Finally, let $P\left(\bar{u},\bar{v}\right)\subseteq \mathbb{R}^n$ be the \emph{primal optimal set} and $Q\left(\bar{u},\bar{v}\right)\subseteq \mathbb{R}^m$ be the \emph{dual optimal set} given by 

\begin{alignat*}{3}
    P\left(\bar{u},\bar{v}\right)&:= \arg\min_{x\in \mathbb{R}^n}\left\{f(x,\bar{u})-\mydot{\bar{v}}{x}\right\}&&=\arg\max_{x\in \mathbb{R}^n} \left\{\mydot{\bar{v}}{x} -f(x,\bar{u})\right\}\\
    Q\left(\bar{u},\bar{v}\right)&:=\arg\min_{y\in\mathbb{R}^m}\left\{ f^*(\bar{v},y) - \mydot{\bar{u}}{y}\right\}&&=\arg\max_{y\in \mathbb{R}^m}\left\{\mydot{\bar{u}}{y}-f^*(\bar{v},y)\right\}.
    \end{alignat*}
\end{definition}

\subsection{General weak duality and optimality conditions}\label{perturbationsection:weak}

The following theorem gives various flavors of weak duality and some optimality conditions that do not require any requirements on $f$.

\begin{restatable}{theorem}{GeneralWeakDual}\label{newGWDOP}
For any $f:\mathbb{R}^n\times \mathbb{R}^m\to \overline{\mathbb{R}}$ and $(\bar{v},\bar{u})\in \mathbb{R}^n\times \mathbb{R}^m$ we have that $q_{\bar{u}}$ is convex, \[f^*(\bar{v},\bar{y}) - \mydot{\bar{u}}{\bar{y}}=\sup_{x\in \mathbb{R}^n}\left\{\mydot{\bar{v}}{x}-\mydot{\bar{u}}{\bar{y}} -l(x,\bar{y})  \right\}\quad \forall \bar{y}\in \mathbb{R}^m\]  and  $
    q_{\bar{u}}(\bar{v}):=\inf_{y\in \mathbb{R}^m}\sup_{x\in \mathbb{R}^n}\left\{\mydot{\bar{v}}{x}-\mydot{\bar{u}}{y} -l(x,y)  \right\}$.

In addition, 
\begin{alignat}{5}
   \operatorname{opt}_P(\bar{u}, \bar{v})&= p_{\bar{v}}(\bar{u})&\geq \operatorname{cl} \operatorname{co}p_{\bar{v}}(\bar{u})
    &&\geq -q_{\bar{u}}(\bar{v})&=\phantom{-}\operatorname{opt}_D(\bar{u}, \bar{v}),\label{newweakGNLPduality}\\
   -\operatorname{opt}_P(\bar{u}, \bar{v})&= -p_{\bar{v}}(\bar{u}) &\leq
\phantom{\operatorname{co}}\operatorname{cl} q_{\bar{u}}(\bar{v})&& \leq \phantom{-} q_{\bar{u}}(\bar{v})&=-\operatorname{opt}_D(\bar{u}, \bar{v}).\label{newdualweakGNLPduality}
\end{alignat}

Finally, for any $(\bar{v},\bar{u})\in \mathbb{R}^n\times \mathbb{R}^m$ and $(\bar{x},\bar{y})\in \mathbb{R}^n\times \mathbb{R}^m$ 
\begin{equation}\label{newGNLP:optcond1}
    \bar{x}\in P\left(\bar{u},\bar{v}\right),\quad
    \bar{y}\in Q\left(\bar{u},\bar{v}\right),\quad \text{and}\quad
    p_{\bar{v}}(\bar{u})=
-q_{\bar{u}}(\bar{v})
\end{equation}
is equivalent to 
\begin{equation}\label{newGNLP:optcond2}
    f(\bar{x},\bar{u})-\mydot{\bar{v}}{\bar{x}}=l(\bar{x},\bar{y})+\mydot{\bar{u}}{\bar{y}} - \mydot{\bar{v}}{\bar{x}} =\mydot{\bar{u}}{\bar{y}}-f^*(\bar{v},\bar{y})\in \mathbb{R}.
\end{equation}

\end{restatable}
\pointtoproof{newdualweakGNLPduality:proof}

In particular, we have that both \eqref{newweakGNLPduality} and \eqref{newdualweakGNLPduality} yield weak duality $\operatorname{opt}_P\geq \operatorname{opt}_D$ for primal-dual pair \eqref{baseprimaldualpair}, but  \eqref{newweakGNLPduality} allows for considering parametric limits on the primal value function $p_v$ while \eqref{newdualweakGNLPduality} allows the same on the (negative of) the dual value function $q_u$. Through the equivalence between \eqref{newGNLP:optcond1} and \eqref{newGNLP:optcond2} Theorem~\ref{newGWDOP} yields optimality conditions based on the Lagrangian function $l$. However, to obtain true \emph{saddle-point} conditions we need some convexity conditions on $f$ as described in the following proposition.

\begin{restatable}{proposition}{partialconvex}\label{partialconvex:label}
For any proper $f:\mathbb{R}^n\times \mathbb{R}^m\to \overline{\mathbb{R}}$ such that $f(x,u)$ is closed and convex in $u$, and for any $(\bar{v},\bar{u})\in \mathbb{R}^n\times \mathbb{R}^m$ we have
\[
 f(\bar{x},\bar{u})-\mydot{\bar{v}}{\bar{x}}=\sup_{y\in \mathbb{R}^{m}}\left\{l(\bar{x},y)+\mydot{\bar{u}}{ y}-\mydot{\bar{v}}{\bar{x}}\right\}\quad \forall \bar{x}\in \mathbb{R}^n\]
  and  $
    p_{\bar{v}}(\bar{u}):=\inf_{x\in \mathbb{R}^n}\sup_{y\in \mathbb{R}^{m}}\left\{l({x},y)+\mydot{\bar{u}}{ y}-\mydot{\bar{v}}{{x}}\right\}$.

In addition, 
\eqref{newGNLP:optcond2} is equivalent to
\begin{subequations}\label{newGNLP:optcond3}
\begin{align}
    \bar{x}&\in \arg\min_{x\in \mathbb{R}^n}\left\{ l({x},\bar{y}) +\mydot{\bar{u}}{\bar{y}}-\mydot{\bar{v}}{{x}}\right\},\\
    \bar{y}&\in \arg\max_{y\in\mathbb{R}^m}\left\{l(\bar{x},y)+\mydot{\bar{u}}{ y}-\mydot{\bar{v}}{\bar{x}}\right\}.
\end{align}
\end{subequations}
\end{restatable}
\pointtoproof{partialconvex:proof}

\subsection{Convex strong duality}\label{perturbationsection:strong}

The following theorem gives various flavors of strong duality and characterizations of the optimal sets under convexity of $f$.

\begin{restatable}{theorem}{CSD}\label{NewConvexStrongDual}
If $f\in \Gamma\left(\mathbb{R}^n\times \mathbb{R}^m\right)$ and $f(x,u)$  is closed in $u$, then $l(x,y)$ is convex in $x$ and concave in $y$
    and   \eqref{newGNLP:optcond3} is equivalent to
    \begin{subequations}\label{newGNLP:optcond4}
    \begin{align}
        0&\in \partial \left(l\left(\cdot, \bar{y}\right)+\mydot{\bar{u}}{\bar{y}}-\mydot{\bar{v}}{\cdot}\right)(\bar{x})\\
        0&\in \partial\left(-l\left(\bar{x}, \cdot\right)-\mydot{\bar{u}}{\cdot}+\mydot{\bar{v}}{\bar{x}}\right)(\bar{y}).
        \end{align}
    \end{subequations}
In addition, if $f\in \Gamma_0\left(\mathbb{R}^n\times \mathbb{R}^m\right)$ we have 
\begin{enumerate}
\item {\bf Optimal sets:} For all $(\bar{v},\bar{u})\in \mathbb{R}^n\times \mathbb{R}^m$ we have that $p_{\bar{v}}$ and $q_{\bar{u}}$ are convex, and 
\begin{subequations}\label{newGNLP:args}
     \begin{align}
      P\left(\bar{u},\bar{v}\right)&=\partial q_{\bar{u}}(\bar{v})\\
                 Q\left(\bar{u},\bar{v}\right)&=\partial p_{\bar{v}}(\bar{u}).
     \end{align}
\end{subequations}
   \item {\bf \emph{Primal-asymptotic} strong duality:} If $p_{\bar{v}}$ is proper, then
      $p_{\bar{v}}(\bar{u})\geq 
    \operatorname{cl} p_{\bar{v}}(\bar{u})
    = -q_{\bar{u}}(\bar{v})$ for all  $\bar{u}\in \mathbb{R}^m$.

      \item {\bf \emph{Primal-uniform} strong duality:}  If $p_{\bar{v}}$ is proper and closed, then
      $p_{\bar{v}}(\bar{u}) = -q_{\bar{u}}(\bar{v})$ for all  $\bar{u}\in \mathbb{R}^m$.

       \item {\bf \emph{Dual-asymptotic} strong duality:} If  $q_{\bar{u}}$ is proper, then 
       $-p_{\bar{v}}(\bar{u}) =
\operatorname{cl} q_{\bar{u}}(\bar{v})$ for all  $\bar{v}\in \mathbb{R}^n$.
     
              \item {\bf \emph{Dual-uniform} strong duality:} If  $q_{\bar{u}}$ is proper and closed, then 
              $-p_{\bar{v}}(\bar{u}) =
 q_{\bar{u}}(\bar{v})$ for all  $\bar{v}\in \mathbb{R}^n$.
   
\end{enumerate}
\end{restatable}
\pointtoproof{NewConvexStrongDual:proof}

Theorem~\ref{NewConvexStrongDual} implicitly contains three classes of qualification conditions beyond the baseline $f\in \Gamma_0\left(\mathbb{R}^n\times \mathbb{R}^m\right)$.

First, requiring $p_{\bar{v}}$ or $q_{\bar{u}}$ to be proper yield asymptotic versions of strong duality where a zero duality gap is achieved through taking limits on the primal or dual perturbations $\bar{u}$ or $\bar{v}$. 

Second, requiring $p_{\bar{v}}$ or $q_{\bar{u}}$ to be proper and closed yields non-asymptotic versions of strong duality where zero duality gap is achieved for all perturbations. 

Third, non-emptyness of the sub-differentials in \eqref{newGNLP:args} yields attainment of the primal or dual optimal values. 

Describing these qualification conditions through properties of $p_{\bar{v}}$ and $q_{\bar{u}}$ instead of more typical conditions (e.g. Slater-like conditions) allows for generic and possibly ad-hoc conditions that exploit the structure of particular instantiations of these functions (and indirectly of $f$). However, these generic conditions are also implied by more typical conditions, which can be a better trade-off between applicability and simplicity. \cref{strongattainmentcoro:label} and \cref{strongattainmentcorodual:label} below gives a Slater-like conditions that we will use in our study of composite functions. 

\begin{definition}
$h:\mathbb{R}^n\to \mathbb{R}\cup\{+\infty\}$ is \emph{piecewise linear-quadratic} (PWLQ) if and only if
\[
h(x)=\begin{cases} \frac{1}{2}\mydot{x}{A^ix}+\mydot{q^i}{x} -u_i&x\in P_i, i\in I\\+\infty&\text{o.w.}\end{cases},
\]
where $I$ is finite and for each $i\in I$, $P_i$ is a polyhedron, $A^i\in \mathbb{R}^{n\times n}$, $q^i\in\mathbb{R}^n$ and $u_i\in \mathbb{R}$. 
\end{definition}

\begin{restatable}{corollary}{strongattainmentcoro}[Primal Slater Condition]\label{strongattainmentcoro:label}
Let  $f\in \Gamma_0\left(\mathbb{R}^n\times \mathbb{R}^m\right)$, $\bar{v}\in \mathbb{R}^n$, $\bar{u}\in \mathbb{R}^m$, and $p_{\bar{v}}:\mathbb{R}^m\to \overline{\mathbb{R}}$, $q_{\bar{u}}:\mathbb{R}^n\to\overline{\mathbb{R}}$ be the functions given by \eqref{pqdef}. Then, 
\begin{itemize}
    \item $q_{\bar{u}}$ is proper if $\bar{u}\in \operatorname{dom}(p_0)$, 
    \item $q_{\bar{u}}$ is closed if $\bar{u}\in \operatorname{rint} \operatorname{dom}(p_0)$, or
 $\bar{u}\in  \operatorname{dom}(p_0)$ and $f$ is PWLQ,
 and 
 \item $\partial p_{\bar{v}}(\bar{u})\neq\emptyset$ if $\bar{u}\in \operatorname{rint} \operatorname{dom}(p_0)$ and $p_{\bar{v}}(\bar{u})>-\infty$, or
 $\bar{u}\in  \operatorname{dom}(p_0)$ and $f$ is PWLQ, 
\end{itemize}

In particular, if  either  $\bar{u}\in \operatorname{rint} \operatorname{dom}(p_0)$, or
 $\bar{u}\in  \operatorname{dom}(p_0)$ and $f$ is PWLQ, then for any $\bar{v}\in \mathbb{R}^n$ we have
\begin{equation}\label{dual:attainment:eq}
    \exists \bar{y}\in \mathbb{R}^m\quad\text{s.t}\quad \mydot{\bar{u}}{\bar{y}}-f^*\left(\bar{v},\bar{y}\right)=-q_{\bar{u}}\left(\bar{v}\right)=p_{\bar{v}}\left(\bar{u}\right)\in \mathbb{R}\cup\{-\infty\}.
    \end{equation}
\end{restatable}
\pointtoproof{strongattainmentcoro:label:proof}
Note that \eqref{dual:attainment:eq} implies that strong duality holds for primal-dual pair \eqref{generalprimaldualpair} and the dual optimal value is attained if finite.

\begin{restatable}{corollary}{strongattainmentcorodual}[Dual Slater Condition]\label{strongattainmentcorodual:label}
Let  $f\in \Gamma_0\left(\mathbb{R}^n\times \mathbb{R}^m\right)$, $\bar{v}\in \mathbb{R}^n$, $\bar{u}\in \mathbb{R}^m$, and $p_{\bar{v}}:\mathbb{R}^m\to \overline{\mathbb{R}}$, $q_{\bar{u}}:\mathbb{R}^n\to\overline{\mathbb{R}}$ be the functions given by \eqref{pqdef}. Then, 
\begin{itemize}
    \item $p_{\bar{v}}$ is proper if $\bar{v}\in \operatorname{dom}(q_0)$, 
    \item $p_{\bar{v}}$ is closed  if $\bar{v}\in \operatorname{rint} \operatorname{dom}(q_0)$, or
 $\bar{v}\in  \operatorname{dom}(q_0)$ and $f$ is PWLQ, and
 \item $\partial q_{\bar{u}}(\bar{v})\neq\emptyset$ if $\bar{v}\in \operatorname{rint} \operatorname{dom}(q_0)$ and $q_{\bar{u}}(\bar{v})>-\infty$, or
 $\bar{v}\in  \operatorname{dom}(q_0)$ and $f$ is PWLQ.
\end{itemize}

In particular, if  either  $\bar{v}\in \operatorname{rint} \operatorname{dom}(q_0)$, or
 $\bar{v}\in  \operatorname{dom}(q_0)$ and $f$ is PWLQ, then for any $\bar{u}\in \mathbb{R}^m$ we have
\begin{equation}\label{primal:attainment:eq}\exists \bar{x}\in \mathbb{R}^n\quad\text{s.t}\quad f\left(\bar{x},\bar{u}\right)-\mydot{\bar{v}}{\bar{x}}=p_{\bar{v}}\left(\bar{u}\right)=-q_{\bar{u}}\left(\bar{v}\right)\in \mathbb{R}\cup\{+\infty\}.\end{equation}
\end{restatable}
\pointtoproof{strongattainmentcorodual:label:proof}
Note that \eqref{primal:attainment:eq} implies that 
strong duality holds for primal-dual pair \eqref{generalprimaldualpair} and the primal optimal value is attained if finite.

\subsection{Generalized nonlinear programming}\label{GNLP:sec}

The generalized nonlinear programming (GNLP) framework of \cite{rockafellar2023augmented} follows from the results in \cite[Section 11.H]{rockafellar2009variational} by taking the specific function $f$ described in the following lemma. A very simple extension of the GNLP framework allows us to cover composite functions in their full generality (see \cref{gnlplagrangianlemma}). Finally, we note that the GNLP framework has theoretical and practical implications that go beyond the scope of this paper (e.g. see \cite{rockafellar2023augmented,rockafellar2023convergence}).

\begin{lemma}[\cite{rockafellar2023augmented}]\label{simplegnlplagrangianlemma}
For $f_0:\mathbb{R}^n\to \overline{\mathbb{R}}$, $F:\mathbb{R}^n\to \mathbb{R}^m$ and $g:\mathbb{R}^m\to \overline{\mathbb{R}}$  let $f:\mathbb{R}^n\times \mathbb{R}^m\to \overline{\mathbb{R}}$ be given by 
\[f(x,u):=f_0(x) +g(F(x)+u)\]
and $l:\mathbb{R}^n\times \mathbb{R}^m\to \overline{\mathbb{R}}$ be given by
\[
l(x,y):= -\left(f(x,\cdot)\right)^*(y).\]
Then $l(x,y)= f_0(x)+\mydot{y}{F(x)} -g^*(y)$ for all $(x,y)\in \mathbb{R}^n\times \mathbb{R}^m$.
\end{lemma}

\section{Composite functions}\label{convex:comp:section}

To cover composite functions in the full generality of  \cite{burke2021study} we only need to extend the GNLP framework of Section~\ref{GNLP:sec} to allow $F$ to take values in $\mathbb{R}^m\cup\{+\infty_\bullet\}$ where $+\infty_\bullet$ is an \emph{infinite element} for $\mathbb{R}^m$ as follows.

\begin{definition}\label{firstkdef}
Let $+\infty_\bullet$ be an \emph{infinite element} for $\mathbb{R}^m$. 
For $F:\mathbb{R}^n\to \mathbb{R}^m\cup\{+\infty_\bullet\}$ we let $\operatorname{dom}(F):=\left\{x\in \mathbb{R}^n\,:\, F(x)\neq+\infty_\bullet \right\}$ and $\operatorname{rge}(F)=F\left(\operatorname{dom}(F)\right)$.
\end{definition}

Here, infinite element $+\infty_\bullet$ is intended to represent all values for $F$ not in $\mathbb{R}^m$. For instance, we can redefine any $H:\mathbb{R}^n\to \left(\mathbb{R}\cup\{-\infty, +\infty\}\right)^m$ to a function  $F:\mathbb{R}^n\to \mathbb{R}^m\cup\{+\infty_\bullet\}$ by letting
\[
F(x):=\begin{cases}H(x) &H(x)\in \mathbb{R}^m\\+\infty_\bullet&\text{o.w.}\end{cases}.
\]
In \cref{secondkdef}, $+\infty_\bullet$ will correspond to a \emph{largest element} for a specific ordering.

\cref{simplegnlplagrangianlemma} then extends to the following lemma.

\begin{restatable}{lemma}{GNLPLagrangianLemma}\label{gnlplagrangianlemma}
Let  
$f_0:\mathbb{R}^n\to \overline{\mathbb{R}}$, $F:\mathbb{R}^n\to \mathbb{R}^m\cup\{+\infty_\bullet\}$ and $g:\mathbb{R}^m\to \overline{\mathbb{R}}$  let $f:\mathbb{R}^n\times \mathbb{R}^m\to \overline{\mathbb{R}}$, $\mydot{y}{F}:\mathbb{R}^n\to \overline{\mathbb{R}}$ and $l:\mathbb{R}^n\times \mathbb{R}^m\to \overline{\mathbb{R}}$ be given by
\begin{subequations}
\begin{alignat}{3}
f(x,u)&:=\begin{cases}f_0(x)+g\left(F(x)+u\right)&x\in \operatorname{dom}(F)\\+\infty&\text{o.w.} \end{cases},&\label{extendedcompositedef}\\ 
  \mydot{y}{F}(x)&:=\begin{cases}\mydot{y}{F(x)}&x\in \operatorname{dom}(F)\\+\infty&\text{o.w.}\end{cases},\label{extendedscalarizationdef}\\
l(x,y)&:= -\left(f(x,\cdot)\right)^*(y).
\end{alignat}
\end{subequations}
Then
\begin{subequations}\label{gnlplagrangianlemma:char}
\begin{align}
l(x,y)&= f_0(x)+\mydot{y}{F(x)} -g^*(y),\label{gnlplagrangianlemma:char:a}\\
f^*(v,y)&=\left(f_0 +\mydot{y}{ F}\right)^*(v)+g^*(y).\label{gnlplagrangianlemma:char:b}
\end{align}
\end{subequations}
\end{restatable}
\pointtoproof{gnlplagrangianlemma:proof}

Finally, applying the results from Section~\ref{perturbationsection} to the function $f$ from \cref{gnlplagrangianlemma} we get the following characterization of the convex conjugate of composite functions. 

\begin{theorem}\label{simplecomposite}
For $f_0:\mathbb{R}^n\to \overline{\mathbb{R}}$, $F:\mathbb{R}^n\to \mathbb{R}^m\cup\{+\infty_\bullet\}$ and $g:\mathbb{R}^m\to \overline{\mathbb{R}}$, let $\mydot{y}{F}:\mathbb{R}^n\to \overline{\mathbb{R}}$ given by \eqref{extendedscalarizationdef},  $\rho:\mathbb{R}^n\to \overline{\mathbb{R}}$ be given by
\begin{equation}\label{rhodef}
    \rho(\bar{v}):=\inf_{y \in \mathbb{R}^m}\left(f_0 +\mydot{y}{ F}\right)^*(\bar{v})+g^*(y),
\end{equation}
and $f_0 + g \circ F:\mathbb{R}^n\to \overline{\mathbb{R}}$ be given by 
\begin{equation}\label{generalcompositef}
 (f_0 + g \circ F)(x):=\begin{cases}f_0(x)+g\left(F(x)\right)&x\in \operatorname{dom}(F)\\+\infty&\text{o.w.} \end{cases}.
\end{equation}
Then, $\rho$ is convex, for any $\bar{v}\in \mathbb{R}^n$
\begin{equation}\label{compositeineq}
    \left(f_0 + g \circ F\right)^*(\bar{v})\leq \operatorname{cl}\rho\left(\bar{v}\right),
\end{equation}
and for any $(\bar{v},\bar{x},\bar{y})\in  \mathbb{R}^n\times \mathbb{R}^n\times \mathbb{R}^m$ 
\begin{subequations}
\label{newcomposite:eqcondone}
\begin{align}
\label{newcomposite:eqcondone:a}
    (f_0+ g \circ F)(\bar{x})+(f_0+g \circ F)^*(\bar{v})&=\mydot{\bar{v}}{\bar{x}}\\
    \label{newcomposite:eqcondone:b}
    (f_0+g \circ F)^*(\bar{v})&=(f_0+ \mydot{\bar{y}}{ F})^*(\bar{v})+g^*(\bar{y})
\end{align}
\end{subequations}
is equivalent to 
\begin{subequations}
\label{newcomposite:eqcondtwo}
\begin{align}g(F(\bar{x}))+g^*(\bar{y})&=\mydot{\bar{y}}{F(\bar{x})}\label{newcomposite:eqcondtwo:a}\\
(f_0+\mydot{\bar{y}}{F})(\bar{x})+(f_0+\mydot{\bar{y}}{ F})^*(\bar{v})&=\mydot{\bar{v}}{\bar{x}}.\label{newcomposite:eqcondtwo:b}
\end{align}
\end{subequations}

Finally, let $f:\mathbb{R}^n\times \mathbb{R}^m\to \overline{\mathbb{R}}$ be given by \eqref{extendedcompositedef} and
\begin{align}
\notag
    U:&=\left\{u\in \mathbb{R}^m\,:\, \inf_{x\in \mathbb{R}^n}f(x,u)<+\infty\right\}\\
    &=\operatorname{dom}(g)-F\left(\operatorname{dom}(f_0)\cap \operatorname{dom}(F)\right).\label{Udef}
\end{align}
Assuming that  $f\in \Gamma_0\left(\mathbb{R}^n\times \mathbb{R}^m\right)$ we have:
\begin{enumerate}
    \item If $0\in U$, then  $\rho\in \Gamma\left(\mathbb{R}^n\right)$.
 \item If $\rho\in \Gamma\left(\mathbb{R}^n\right)$ (e.g. if $0\in U$), then  \eqref{compositeineq} holds at equality.
    \item If $0\in \operatorname{rint}(U)$ or $0\in U$ and $f$ is PWLQ, then  $\rho\in \Gamma_0\left(\mathbb{R}^n\right)$, and the infimum in \eqref{rhodef} is a minimum (possibly with $\rho(\bar{v})=\left(f_0 +\mydot{y}{ F}\right)^*(\bar{v})+g^*(y)=+\infty$ for all $y \in \mathbb{R}^m$).
   
\end{enumerate}

\end{theorem}
\begin{proof}{\textbf{Proof}\hspace*{0.5em}}
First note that by using the characterization of $f^*$ from Lemma~\ref{gnlplagrangianlemma} we have  $\rho\left(\bar{v}\right)=q_0\left(\bar{v}\right)$ and $\left(f_0 + g \circ F\right)^*(\bar{v})=-p_{\bar{v}}(0)$ for all $\bar{v}\in \mathbb{R}^n$ where  $q_{\bar{u}}$ and $p_{\bar{v}}$ are defined in \eqref{pqdef}. Then convexity of $\rho$ and \eqref{compositeineq} follows from Theorem~\ref{newGWDOP} (specifically  \eqref{newdualweakGNLPduality}).

The characterizations from Lemma~\ref{gnlplagrangianlemma} also imply that \eqref{newGNLP:optcond1} is equivalent to  \eqref{newcomposite:eqcondone}  and that \eqref{newGNLP:optcond2} is equivalent to \eqref{newcomposite:eqcondtwo} as follows.

For the equivalence between \eqref{newGNLP:optcond1} and \eqref{newcomposite:eqcondone} note that $\bar{x}\in P\left(0,\bar{v}\right)$ is equivalent to \eqref{newcomposite:eqcondone:a}, $\bar{y}\in Q\left(0,\bar{v}\right)$ is equivalent to $\rho(\bar{v})=\left(f_0 +\mydot{\bar{y}}{ F}\right)^*(\bar{v})+g^*(\bar{y})$, and $p_{\bar{v}}(0)=
-q_{0}(\bar{v})$ is equivalent to $\left(f_0 + g \circ F\right)^*(\bar{v})=\rho(\bar{v})$. The equivalence then follows because \eqref{rhodef} and \eqref{compositeineq} always imply
\[\left(f_0 + g \circ F\right)^*(\bar{v})\leq \rho(\bar{v})\leq \left(f_0 +\mydot{\bar{y}}{ F}\right)^*(\bar{v})+g^*(\bar{y}).\]

The equivalence between  \eqref{newGNLP:optcond2} and \eqref{newcomposite:eqcondtwo} follows by noting that under the characterization from Lemma~\ref{gnlplagrangianlemma}, \eqref{newGNLP:optcond2} becomes
\begin{align*}\mydot{\bar{v}}{\bar{x}}-(f_0+ g \circ F)(\bar{x})&=g^*(\bar{y})-f_0(\bar{x})-\mydot{\bar{y}}{ F}(\bar{x})+\mydot{\bar{v}}{\bar{x}}\\&=g^*(\bar{y})+\left(f_0 +\mydot{\bar{y}}{ F}\right)^*(\bar{v}).\end{align*}

The equivalence between \eqref{newcomposite:eqcondone} and \eqref{newcomposite:eqcondtwo} then follows from the corresponding equivalence between  \eqref{newGNLP:optcond1} and \eqref{newGNLP:optcond2} in Theorem~\ref{newGWDOP}.

The remaining results follow from Corollary~\ref{strongattainmentcoro:label} by noting that 
$U=\operatorname{dom}(p_0)$.

 \Halmos \end{proof}

Next, in  Section~\ref{sub:diff:section} we show how the equivalence between \eqref{newcomposite:eqcondone} and \eqref{newcomposite:eqcondtwo} can be used to deduce a chain rule for the sub-differential of $g\circ F$. 
After that in Section~\ref{my:additive:results:section} how standard results on the sum of two functions can be used to refine both results (e.g. to write $\left(f_0 +\mydot{y}{ F}\right)^*$ in \eqref{rhodef} as a function of $f_0^*$ and $\left(\mydot{y}{ F}\right)^*$).

\subsection{Sub-differential chain rule}\label{sub:diff:section}

Combining the equivalence between \eqref{newcomposite:eqcondone} and \eqref{newcomposite:eqcondtwo} in \cref{simplecomposite} with Fenchel's inequality (e.g. \cite[Proposition 11.3]{rockafellar2009variational}) we also get the sub-differential chain rule for convex composite functions given in \cref{sub:diff:calc:coro} below.

\begin{corollary}\label{sub:diff:calc:coro}
    For $f_0:\mathbb{R}^n\to \overline{\mathbb{R}}$, $F:\mathbb{R}^n\to \mathbb{R}^m\cup\{+\infty_\bullet\}$ and $g:\mathbb{R}^m\to \overline{\mathbb{R}}$, let $\mydot{y}{F}:\mathbb{R}^n\to \overline{\mathbb{R}}$ given by \eqref{extendedscalarizationdef}, 
  $f_0 + g \circ F:\mathbb{R}^n\to \overline{\mathbb{R}}$ be given by \eqref{generalcompositef}, $f:\mathbb{R}^n\times \mathbb{R}^m\to \overline{\mathbb{R}}$ be given by \eqref{extendedcompositedef}, and $U\subseteq \mathbb{R}^m$ be given by \eqref{Udef}.

If  $f\in \Gamma\left(\mathbb{R}^n\times \mathbb{R}^m\right)$, $f(x,u)$ is closed in $u$, and $\operatorname{dom}(F)\neq\emptyset$, then $f_0+g\circ F$ and $g$ are convex, $f_0+\mydot{\bar{y}}{F}$ is convex for any $\bar{y}\in \operatorname{dom}(g^*)$ and for all $\bar{x}\in \operatorname{dom}\left(f_0+g\circ F\right)$ we have 
  \begin{equation}
            \partial \left(f_0+g\circ F\right)\left(\bar{x}\right)\supseteq\bigcup_{\bar{y}\in \partial g(F(\bar{x}))} \partial\left(f_0+\mydot{\bar{y}}{F}\right)(\bar{x})\label{sub:diff:calc:coro:eq}
  \end{equation}  
  In addition, \eqref{sub:diff:calc:coro:eq} holds at equality if $f \in\Gamma_0(\mathbb{R}^n\times \mathbb{R}^m)$ and either $0\in \operatorname{rint}(U)$ or $0\in U$ and $f$ is PWLQ.
\end{corollary}
\begin{proof}{\textbf{Proof}\hspace*{0.5em}}
    Convexity of $f_0+g\circ F$ is direct from the joint convexity of $f$ and convexity of $g$ follows from convexity of $f$ and the assumption $\operatorname{dom}(F)\neq\emptyset$. Convexity of $f_0+\mydot{\bar{y}}{F}$ follows by noting that by \cref{NewConvexStrongDual} and \cref{gnlplagrangianlemma} we have that if $f(x,u)$ is convex in $(x, u)$ and closed in $u$, then 
    $l(x,y)= f_0(x)+\mydot{y}{F(x)} -g^*(y)$
    is convex in $x$.

    Now by Fenchel's inequality (e.g. \cite[Proposition 11.3]{rockafellar2009variational}) we have that
    \begin{itemize}
        \item \eqref{newcomposite:eqcondone:a} is equivalent to $\bar{v}\in  \partial \left(f_0+g\circ F\right)\left(\bar{x}\right)$,
        \item \eqref{newcomposite:eqcondtwo:a} is equivalent to $\bar{y}\in \partial g(F(\bar{x}))$, and
        \item \eqref{newcomposite:eqcondtwo:b} is equivalent to $
        \bar{v}\in\partial\left(f_0+\mydot{\bar{y}}{F}\right)(\bar{x})$.
    \end{itemize}
    Then, \eqref{sub:diff:calc:coro:eq} follows because \eqref{newcomposite:eqcondtwo:a} and \eqref{newcomposite:eqcondtwo:b} imply \eqref{newcomposite:eqcondone:a} by \cref{simplecomposite}.

    Under the additional assumptions \cref{simplecomposite} implies that for any $\bar{x},\bar{v}\in \mathbb{R}^n$ that satisfies \eqref{newcomposite:eqcondone:a} there exists $\bar{y}\in \mathbb{R}^m$ that satisfies \eqref{newcomposite:eqcondone:b}. Equality in \eqref{sub:diff:calc:coro:eq} then follows from the equivalence between \eqref{newcomposite:eqcondone} and \eqref{newcomposite:eqcondtwo}.
 \Halmos \end{proof}

\subsection{Combining with standard additive results}\label{my:additive:results:section}

We can apply standard results for $\left(f_0 +\mydot{y}{ F}\right)^*$ and $\partial\left(f_0+\mydot{\bar{y}}{F}\right)(\bar{x})$ to further refine  \cref{simplecomposite}  and \cref{sub:diff:calc:coro} in \cref{simplecompositeadditive} below.  Such standard results require $\mydot{\bar{y}}{F}\in \Gamma(\mathbb{R}^n)$, which the following lemma shows can be deduced from properties of the portion of $f$ that does not depend on $f_0$.

\begin{lemma}\label{simplecompositeadditive:lemma}
    For $f_0:\mathbb{R}^n\to \overline{\mathbb{R}}$, $F:\mathbb{R}^n\to \mathbb{R}^m\cup\{+\infty_\bullet\}$ and $g:\mathbb{R}^m\to \overline{\mathbb{R}}$, let $\mydot{y}{F}:\mathbb{R}^n\to \overline{\mathbb{R}}$ given by \eqref{extendedscalarizationdef}, 
   $f:\mathbb{R}^n\times \mathbb{R}^m\to \overline{\mathbb{R}}$ be given by \eqref{extendedcompositedef} and $h:\mathbb{R}^n\times \mathbb{R}^m\to \overline{\mathbb{R}}$ be given by
   \begin{equation}
          h(x,u):=\begin{cases}g\left(F(x)+u\right)&x\in \operatorname{dom}(F)\\+\infty&\text{o.w.} \end{cases}\label{extendedcompositedefpartial}
     \end{equation}
so that $f=f_0+h$ (or equivalently so that $h=f$ with $f_0\equiv 0$).

If $h\in\Gamma\left(\mathbb{R}^n\times \mathbb{R}^m\right)$ and $h(x,u)$ is closed in $u$,  then  $\mydot{\bar{y}}{F}\in \Gamma(\mathbb{R}^n)$ for all $\bar{y}\in \operatorname{dom}(g^*)$.
\end{lemma}
\begin{proof}{\textbf{Proof}\hspace*{0.5em}}
By \cref{NewConvexStrongDual} and \cref{gnlplagrangianlemma} (applied to $f$ with $f_0\equiv 0$) we have that if $h(x,u)$ is convex in $(x, u)$ and closed in $u$, then 
    $l(x,y)= \mydot{y}{F(x)} -g^*(y)$
    is convex in $x$. In addition, $h\in\Gamma\left(\mathbb{R}^n\times \mathbb{R}^m\right)$ implies $g\not\equiv +\infty$ and hence $g^*(y)>-\infty$ for all $y\in \mathbb{R}^m$. Then $\operatorname{dom}(g^*)=\{y\in \mathbb{R}^m\,:\, g^*(y)\in \mathbb{R}\}$. We also have that $h\in\Gamma\left(\mathbb{R}^n\times \mathbb{R}^m\right)$ implies  $\operatorname{dom}(F)\neq\emptyset$. The result follows by noting that by its definition in  \eqref{extendedscalarizationdef}, $\mydot{\bar{y}}{F}$ is proper if and only if $\operatorname{dom}(F)\neq\emptyset$.
 \Halmos \end{proof}

\begin{proposition}\label{simplecompositeadditive}
    For $f_0:\mathbb{R}^n\to \overline{\mathbb{R}}$, $F:\mathbb{R}^n\to \mathbb{R}^m\cup\{+\infty_\bullet\}$ and $g:\mathbb{R}^m\to \overline{\mathbb{R}}$, let $\mydot{y}{F}:\mathbb{R}^n\to \overline{\mathbb{R}}$ given by \eqref{extendedscalarizationdef}, 
  $f_0 + g \circ F:\mathbb{R}^n\to \overline{\mathbb{R}}$ be given by \eqref{generalcompositef}, $f:\mathbb{R}^n\times \mathbb{R}^m\to \overline{\mathbb{R}}$ be given by \eqref{extendedcompositedef}, $h:\mathbb{R}^n\times \mathbb{R}^m\to \overline{\mathbb{R}}$  be given by \eqref{extendedcompositedefpartial}, $U\subseteq \mathbb{R}^m$ be given by \eqref{Udef}, $\rho:\mathbb{R}^n\to \overline{\mathbb{R}}$ be given by \eqref{rhodef} and $\tilde{\rho}:\mathbb{R}^n\to \overline{\mathbb{R}}$ be given by
\begin{equation}\label{rhotildedef}
    \tilde{\rho}(\bar{v}):=\inf_{y \in \mathbb{R}^m,\;w\in \mathbb{R}^n}f_0^*({w})+\left(\mydot{y}{ F}\right)^*(\bar{v}-w)+g^*(y),
\end{equation}

If $f_0\in \Gamma(\mathbb{R}^n)$,  $f,h\in \Gamma\left(\mathbb{R}^n\times \mathbb{R}^m\right)$, $h(x,u)$ is closed in $u$, 
then $f_0+g\circ F$ is convex, and $f_0+\mydot{\bar{y}}{F}, \mydot{\bar{y}}{F}\in \Gamma(\mathbb{R}^n)$ and for all $\bar{y}\in \operatorname{dom}(g^*)$. In addition,  for all $\bar{v}\in \mathbb{R}^n$ we have 
\begin{equation}\label{simplecompositeadditive:conj}
    \left(f_0 + g \circ F\right)^*(\bar{v})\leq \rho\left(\bar{v}\right)\leq \tilde{\rho}\left(\bar{v}\right),
\end{equation}
 and for all $\bar{x}\in \operatorname{dom}\left(f_0+g\circ F\right)$ we have 
   \begin{align}
      \partial \left(f_0+g\circ F\right)\left(\bar{x}\right)&\supseteq\bigcup_{\bar{y}\in \partial g(F(\bar{x}))} \partial\left(f_0+\mydot{\bar{y}}{F}\right)(\bar{x})\notag\\
      &\supseteq \partial f_0(\bar{x})+\bigcup_{\bar{y}\in \partial g(F(\bar{x}))} \partial\mydot{\bar{y}}{F}(\bar{x}).\label{simplecompositeadditive:subdif}
  \end{align}
  In addition, equality holds throughout \eqref{simplecompositeadditive:conj} and \eqref{simplecompositeadditive:subdif} if additionally $f \in\Gamma_0(\mathbb{R}^n\times \mathbb{R}^m)$,
  \begin{equation}\label{simplecompositeadditive:inf:conv:condion}
    \operatorname{rint}\left(\operatorname{dom}(f_0)\right)\cap\operatorname{rint}\left(\operatorname{dom}(F)\right)\neq \emptyset
\end{equation}
  and either $0\in \operatorname{rint}(U)$ or $0\in U$ and $f$ is PWLQ.
\end{proposition}
\begin{proof}{\textbf{Proof}\hspace*{0.5em}}
  Convexity of $f_0+g\circ F$ is direct from $f\in \Gamma\left(\mathbb{R}^n\times \mathbb{R}^m\right)$. \cref{simplecompositeadditive:lemma} implies $\mydot{\bar{y}}{F}\in \Gamma(\mathbb{R}^n)$ for all $\bar{y}\in \operatorname{dom}(g^*)$. $f\in \Gamma\left(\mathbb{R}^n\times \mathbb{R}^m\right)$ also implies  $\operatorname{dom}(f_0)\cap\operatorname{dom}(F)\neq\emptyset$. Then $f_0\in \Gamma(\mathbb{R}^n)$ and $\mydot{\bar{y}}{F}\in \Gamma(\mathbb{R}^n)$ for all $\bar{y}\in \operatorname{dom}(g^*)$ implies $f_0+\mydot{\bar{y}}{F}\in \Gamma(\mathbb{R}^n)$ for all $\bar{y}\in \operatorname{dom}(g^*)$.

  The first inequality in \eqref{simplecompositeadditive:conj} and its equality under the additional assumptions follows from  \cref{simplecomposite}. The second inequality in \eqref{simplecompositeadditive:conj} and its equality under the additional assumptions follows from  \cite[Theorem 16.4]{rockafellar2015convex}, which is applicable because $f_0, \mydot{\bar{y}}{F}\in \Gamma(\mathbb{R}^n)$ for all $\bar{y}\in \operatorname{dom}(g^*)$.

 The first inclusion in  \eqref{simplecompositeadditive:subdif} and its equality under the additional assumptions follows from  \cref{sub:diff:calc:coro}. The second inclusion in  \eqref{simplecompositeadditive:subdif} and its equality under the additional assumptions follows from \cite[Theorem 23.8]{rockafellar2015convex}, which is also applicable because $f_0, \mydot{\bar{y}}{F}\in \Gamma(\mathbb{R}^n)$ for all $\bar{y}\in \operatorname{dom}(g^*)$ and $\bar{y}\in \partial g(F(\bar{x}))$ implies $\bar{y}\in \operatorname{dom}(g^*)$.
 \Halmos \end{proof}

\section{Connection of Theorem~\ref{simplecomposite} to K-convexity}\label{connection:section}

A standard sufficient condition for convexity of $g\circ F$ is component-wise convexity of $F$ and component-wise monotonicity of $g$ (e.g. \cite{hiriart2006note}). Such condition can be further generalized through the following notions of convexity and monotonicity with regards to a closed convex cone $K$ \cite{burke2021study,gissler2023note, pennanen1999graph}.

\begin{definition}\label{secondkdef}
Given a closed convex cone $K\subseteq \mathbb{R}^m$ we define the ordering 
\[x^1\leq_K x^2\quad\Leftrightarrow\quad x^2-x^1\in K\quad \forall x^1,x^2 \in \mathbb{R}^m.\]
We also extend the ordering to consider the infinite element $+\infty_\bullet$ (cf. \cref{firstkdef}) as its \emph{largest element} such that
\[x\leq_K +\infty_\bullet \quad\forall x\in \mathbb{R}^m.\]
\end{definition}
 
\begin{definition}
Let $K\subseteq \mathbb{R}^m$ be a closed convex cone and  $+\infty_\bullet$ be a \emph{largest element} with respect to $K$ (cf. Definition~\ref{secondkdef}).

We say  $F:\mathbb{R}^n\to\mathbb{R}^m\cup\{+\infty_\bullet\}$ is \emph{$K$-convex} if and only if
\[K\text{-}\operatorname{epi}(F):=\left\{(x,z)\in \operatorname{dom}(F)\times \mathbb{R}^m\,:\,  F(x)\leq_K z \right\}\]
is convex.

We say $g:\mathbb{R}^m\to\overline{\mathbb{R}}$ is \emph{$K$-increasing} if and only if
\[g(w)\leq g(w+k)\quad \forall k\in K,\;w\in \mathbb{R}^m,\]
and is \emph{$K$-increasing restricted to $\operatorname{rge}(F)$} if and only if
\[g(w)\leq g(w+k)\quad \forall k\in K,\;w\in \operatorname{rge}(F).\]

Finally, we let 
\begin{align*}
    \mathcal{K}_F&:=\left\{K\subseteq \mathbb{R}^m\,:\, \text{$K$ is a closed convex cone and $F$ is $K$-convex}\right\}\\
    \mathcal{K}_g&:=\left\{K\subseteq \mathbb{R}^m\,:\, \text{$K$ is a closed convex cone and $g$ is $K$-increasing}\right\}\\
    \mathcal{K}_{g,\operatorname{rng}(F)}&:=\left\{K\subseteq \mathbb{R}^m\,:\, 
    \begin{aligned}
    \text{$K$ is a closed convex cone and}\\
    \text{$g$ is $K$-increasing restricted to $\operatorname{rge}(F)$}\end{aligned}\right\}
\end{align*}
\end{definition}

\begin{restatable}{proposition}{allconvexityprop}\label{allconvexityprop:label}
Let $g:\mathbb{R}^m\to \overline{\mathbb{R}}$, $F:\mathbb{R}^n\to\mathbb{R}^m\cup\{+\infty_\bullet\}$, $g\circ F:\mathbb{R}^n\to\overline{\mathbb{R}}$ be given by
\begin{equation}\label{goFdef}
g\circ F(x):=\begin{cases}g\left(F(x)\right)&x\in\operatorname{dom}(F)\\+\infty&\text{o.w.} \end{cases},
\end{equation}
and $f_0:\mathbb{R}^n\to \overline{\mathbb{R}}$ be such that $f_0\equiv 0$ so that $f:\mathbb{R}^n\times\mathbb{R}^m\to \overline{\mathbb{R}}$ given by \eqref{extendedcompositedef} becomes
\begin{equation}\label{nofzerogeneralcompositef}
f(x,u):=\begin{cases}g\left(F(x)+u\right)&x\in\operatorname{dom}(F)\\+\infty&\text{o.w.} \end{cases}
\end{equation}
and $g\circ F(x)=f(x,0)$ for all $x\in \mathbb{R}^n$.

If  $g\in\Gamma(\mathbb{R}^m)$, then 
\begin{subequations}
    \begin{alignat}{5}
        \mathcal{K}_F&\cap\mathcal{K}_{g,\operatorname{rng}(F)}&&\neq \emptyset&&\Rightarrow \text{$g\circ F$ is convex}\label{weakallconvexitypropequiv}\\
        \mathcal{K}_F&\cap\mathcal{K}_{g}&&\neq \emptyset&&\Rightarrow \text{$f$ is convex}\label{allconvexitypropequiv}
    \end{alignat}
\end{subequations}
If  $g\in\Gamma_0(\mathbb{R}^m)$  and $\operatorname{dom}(g)\cap\operatorname{rge}(F) \neq \emptyset$, then \eqref{allconvexitypropequiv} is an equivalence.
\end{restatable}
\pointtoproof{allconvexityprop:proof}

As illustrated in \cite[Example 58]{gissler2023note}, convexity of $g\circ F$ can be deduced through $\mathcal{K}_F\cap\mathcal{K}_{g,\operatorname{rng}(F)}\neq \emptyset$ for cases in which $\mathcal{K}_F\cap\mathcal{K}_{g}\neq \emptyset$ does not hold. This is not surprising given that the latter condition implies convexity of $f$, which is stronger than just convexity of $g\circ F$. Based on condition $\mathcal{K}_F\cap\mathcal{K}_{g,\operatorname{rng}(F)}\neq \emptyset$,   \cite{burke2021study} also studies characterizations of $(g\circ F)^*$ as in Theorem~\ref{simplecomposite}. The first version of this characterization is \cite[Theorem 4]{burke2021study} that focuses on inequality \eqref{compositeineq} from Theorem~\ref{simplecomposite} when $f_0\equiv 0$. In this setting, we can restate \cite[Theorem 4]{burke2021study} as follows.

\begin{theorem}[{\cite[Theorem 4]{burke2021study}}]\label{burketheo}
Let $K\subseteq \mathbb{R}^m$ be a closed convex cone, $g:\mathbb{R}^m\to \overline{\mathbb{R}}$, $F:\mathbb{R}^n\to \mathbb{R}^m\cup\{+\infty_\bullet\}$, $\mydot{y}{F}:\mathbb{R}^n\to \overline{\mathbb{R}}$ be given by \eqref{extendedscalarizationdef}, $g \circ F:\mathbb{R}^n\to \overline{\mathbb{R}}$ be given by \eqref{goFdef} and $\eta:\mathbb{R}^n\to \overline{\mathbb{R}}$ be given by
\begin{equation}\label{etadef}
    \eta(v):=\inf_{y \in-K^\circ}\left(\mydot{y}{ F}\right)^*(v)+g^*(y),
\end{equation}

If  $g\in\Gamma(\mathbb{R}^m)$,  and $K\in \mathcal{K}_F\cap\mathcal{K}_{g,\operatorname{rng}(F)}$, then for any $\bar{v}\in \mathbb{R}^n$
\begin{equation}\label{burke:compositeineq}
    \left(g \circ F\right)^*(\bar{v})\leq \operatorname{cl}\eta\left(\bar{v}\right).
\end{equation}

If  $g\in\Gamma_0(\mathbb{R}^m)$,  $K\in \mathcal{K}_F\cap\mathcal{K}_{g}$, $
   \operatorname{dom}(g)\cap  \operatorname{rge}(F)\neq \emptyset$, and $
    \mydot{y}{F}\in \Gamma_0(\mathbb{R}^n)$   for all $y\in -K^\circ$, then \eqref{burke:compositeineq} holds at equality.

Finally, if $g\in\Gamma(\mathbb{R}^m)$,  $K\in \mathcal{K}_F\cap\mathcal{K}_{g,\operatorname{rng}(F)}$ and
\begin{equation}\label{burke:qc}
    \operatorname{rint}\left(\operatorname{dom}(g)-K\right)\cap F\left(\operatorname{rint}\left(\operatorname{dom}(F)\right)\right) \neq \emptyset,
\end{equation}
then $\eta\in\Gamma_0(\mathbb{R}^n)$ and the infimum in \eqref{etadef} is a minimum (possibly with $\eta(v)=\left(\mydot{y}{ F}\right)^*(v)+g^*(y)=+\infty$ for all $y \in-K^\circ$).
\end{theorem}

A chain rule for the sub-differential of $g\circ F$ is also presented in \cite[Corollary 2]{burke2021study}. This result is identical to 
\cref{sub:diff:calc:coro} except for the necessary conditions for \eqref{sub:diff:calc:coro:eq} to hold (i) as is and (ii) as an equality. These conditions are the same as those (i) for  \eqref{burke:compositeineq} to hold and (ii) for  \eqref{burke:compositeineq}  to hold at equality, $\eta$ being closed and then infimum in \eqref{etadef} being attained. Hence, a comparison between \cite[Corollary 2]{burke2021study} and \cref{sub:diff:calc:coro} follows directly from the comparison between \ref{burketheo} and \cref{simplecomposite} in Section~\ref{comparison:sub:section}.

Another corollary of \cref{burketheo} considered in \cite{burke2021study} is \cite[Corollary 3]{burke2021study} which extends the convex conjugate and sub-differential formulas for the case $f_0\not\equiv 0$. We compare \cite[Corollary 3]{burke2021study}  with \cref{simplecompositeadditive} in Section~\ref{additive:comparison:section}.

\subsection{Comparison between Theorems \ref{simplecomposite} and \ref{burketheo}}\label{comparison:sub:section}

Under different requirements, both $\eta:\mathbb{R}^n\to \overline{\mathbb{R}}$ from Theorem~\ref{burketheo} and $\rho:\mathbb{R}^n\to \overline{\mathbb{R}}$
from Theorem~\ref{simplecomposite} (with $f_0\equiv 0$) yield upper bounds on $\left(g \circ F\right)^*$ with $\rho$ being a potentially tighter bound:
\begin{equation}
    \left(g \circ F\right)^*(\bar{v})\leq \operatorname{cl}\rho\left(\bar{v}\right)\leq \operatorname{cl}\eta\left(\bar{v}\right)\quad\forall \bar{v}\in \mathbb{R}^n.
\end{equation}
It is worth noting that the only reason that $\rho$ is potentially stronger than $\eta$ is that the definition of $\eta$ in  \eqref{etadef} restricts the infimum over values $y\in -K^\circ$.  We will further study this in Section~\ref{grepairsection}. 

To compare the requirements of Theorems~\ref{simplecomposite} and \ref{burketheo} we will use the following lemma.

\begin{restatable}{lemma}{firstequivlemmastatement} \label{firstequivlemma}
Let $K\subseteq \mathbb{R}^m$ be a closed convex cone,  $g:\mathbb{R}^m\to \overline{\mathbb{R}}$, $F:\mathbb{R}^n\to \mathbb{R}^m\cup\{+\infty_\bullet\}$, $\mydot{y}{F}:\mathbb{R}^n\to \overline{\mathbb{R}}$ be given by \eqref{extendedscalarizationdef} and $f:\mathbb{R}^n\times\mathbb{R}^m\to \overline{\mathbb{R}}$ given by \eqref{nofzerogeneralcompositef}.

If  $g\in\Gamma_0(\mathbb{R}^m)$,  $K\in \mathcal{K}_F\cap\mathcal{K}_{g}$, $
    \operatorname{dom}(g)\cap\operatorname{rge}(F)\neq \emptyset$, and $
    \mydot{y}{F}\in\Gamma_0(\mathbb{R}^n)$ for all $y\in -K^\circ$,   then $f \in\Gamma_0(\mathbb{R}^n\times \mathbb{R}^m)$.

\end{restatable}
\pointtoproof{firstequivlemma:proof}

With regards to $\operatorname{cl}\rho$ and $\operatorname{cl}\eta$ providing upper bounds on $(g\circ F)^*$ we see that Theorem~\ref{simplecomposite} does not require anything for it to hold, while Theorem~\ref{burketheo} requires the existence of a closed convex cone  $K\in \mathcal{K}_F\cap\mathcal{K}_{g,\operatorname{rng}(F)}$, which is used to define $\eta$. By Proposition~\ref{allconvexityprop:label} these conditions imply convexity of $g\circ F$ which may be considered a natural setting for studying $(g\circ F)^*$. Nonetheless, $(g\circ F)^*$ is convex even if $g\circ F$ is not and Theorem~\ref{simplecomposite} provides $\rho$ as a convex upper bound of $(g\circ F)^*$ even in such non-convex setting.

With regards to $\operatorname{cl}\rho$ and $\operatorname{cl}\eta$ being equal to $(g\circ F)^*$, by Lemma~\ref{firstequivlemma} we have that the requirements of Theorem~\ref{burketheo} imply that $f \in\Gamma_0(\mathbb{R}^n\times \mathbb{R}^m)$  and $0\in \operatorname{dom}(g)-\operatorname{rng}(F)$, which are essentially the requirements of Theorem~\ref{simplecomposite} ($0\in \operatorname{dom}(g)-\operatorname{rng}(F)$ implies $\rho$ being proper, which is the real requirement of Theorem~\ref{simplecomposite}). Combining this with Proposition~\ref{allconvexityprop:label} we see that the requirements of Theorem~\ref{simplecomposite} for
$(g\circ F)^*=\operatorname{cl}\rho$ imply the existence of a closed convex cone  $K\in \mathcal{K}_F\cap\mathcal{K}_{g}$. The requirement $K\in \mathcal{K}_F\cap\mathcal{K}_{g}$ also appears in Theorem~\ref{burketheo} for $(g\circ F)^*=\operatorname{cl}\eta$. Hence the additional flexibility of using $\mathcal{K}_{g,\operatorname{rng}(F)}$ instead of $\mathcal{K}_{g}$ does not seem to provide an advantage when approximating $(g\circ F)^*$ (getting an upper bound does not require it and it is not sufficient to guarantee a tight upper bound). 

With regards to $\eta$ being closed and that the infimum in its definition being attained, we do have that the weaker requirement $K\in \mathcal{K}_F\cap\mathcal{K}_{g,\operatorname{rng}(F)}$ from Theorem~\ref{burketheo} can provide an advantage over the more restrictive $K\in \mathcal{K}_F\cap\mathcal{K}_{g}$ required by Theorem~\ref{simplecomposite} for the same properties holding for $\rho$. However, it could be argued that such properties for $\eta$ are less interesting when $\operatorname{cl}\eta$ is a strict upper bound on $(g\circ F)^*$. For this reason we first restrict our attention to the case where Theorem~\ref{simplecomposite} and \ref{burketheo} imply $(g\circ F)^*=\operatorname{cl}\rho=\operatorname{cl}\eta$ and then in Section~\ref{grepairsection} study how the existence of $K\in \mathcal{K}_F\cap\mathcal{K}_{g,\operatorname{rng}(F)}\setminus \mathcal{K}_{g}$ can still be useful in this context. For simplicity we consider the following joint condition for $(g\circ F)^*=\operatorname{cl}\rho=\operatorname{cl}\eta$. First we assume
\begin{subequations}\label{jointassumptions:baseline}
    \begin{align}
    g&\in\Gamma_0(\mathbb{R}^m),\\ 
    \operatorname{dom}(g)\cap\operatorname{rge}(F) &\neq \emptyset,\\
    f:\mathbb{R}^n\times\mathbb{R}^m\to \overline{\mathbb{R}}\quad&\text{is closed,}\\
    \mydot{y}{F}:\mathbb{R}^n\to \overline{\mathbb{R}}\quad&\text{is proper and closed $\forall y\in \mathbb{R}^m$.}
\end{align}
\end{subequations}
Now, by \cite[Theorem 3]{burke2021study} we have that $K\in \mathcal{K}_{F}$ if and only if
$ \mydot{y}{F}:\mathbb{R}^n\to \overline{\mathbb{R}}$ is convex for all $y\in -K^\circ$. Then, by \cref{allconvexityprop:label} a joint condition for $(g\circ F)^*=\operatorname{cl}\rho=\operatorname{cl}\eta$ is given by \eqref{jointassumptions:baseline} plus either of the following equivalent conditions 
\begin{subequations}\label{jointassumptions}
    \begin{align}
    \label{jointassumptions:a}
   \mathcal{K}_F\cap\mathcal{K}_{g}&\neq \emptyset,\\ 
    f&\text{ is convex}.\label{jointassumptions:b}
\end{align}
\end{subequations}
More precisely, for $(g\circ F)^*=\operatorname{cl}\eta$, the condition is that we use $K\in\mathcal{K}_F\cap\mathcal{K}_{g}$ to define $\eta$ through \eqref{etadef}. To compare the closure and infimum-attainment conditions from Theorem~\ref{simplecomposite} and \ref{burketheo} in this context we will use the following lemma.

\begin{restatable}{lemma}{cscomparisonlemma}\label{cscomparisonlemma:label}
If \eqref{jointassumptions:baseline} and \eqref{jointassumptions:a} (or equivalently \eqref{jointassumptions:b}) holds for $g$ and $F$, then the general condition     
    $0\in \operatorname{rint}(U)$ from Theorem~\ref{simplecomposite} for $f_0\equiv 0$ is equivalent to
    \begin{equation}\label{csgeneralf}
   \exists K\in \mathcal{K}_F\cap\mathcal{K}_{g}\quad\text{s.t}\quad \operatorname{rint}\left(\operatorname{dom}(g)\right)\cap \operatorname{rint}\left(F\left(\operatorname{dom}(F)\right)+K\right)\neq \emptyset,
\end{equation}
the weaker condition for piecewise linear-quadratic functions $0\in U$ from Theorem~\ref{simplecomposite} for $f_0\equiv 0$ is equivalent to
 \begin{equation}\label{cspwlqf}
     \exists K\in \mathcal{K}_F\cap\mathcal{K}_{g}\quad\text{s.t}\quad \left(\operatorname{dom}(g)-K\right)\cap \left(F\left(\operatorname{dom}(F)\right)\right)\neq \emptyset,
\end{equation}
and condition \eqref{burke:qc} from Theorem~\ref{burketheo} is equivalent to
\begin{equation}\label{burke:qc:rewritten}
  K\in \mathcal{K}_F\cap\mathcal{K}_{g}\quad\text{and}\quad   \operatorname{rint}\left(\operatorname{dom}(g)-K\right)\cap F\left(\operatorname{rint}\left(\operatorname{dom}(F)\right)\right)\neq \emptyset,
\end{equation}
\end{restatable}
\pointtoproof{cscomparisonlemma:proof}

As noted in \cite[pages 1335--1336]{burke2021study}, condition \eqref{csgeneralf} matches the Attouch-Br\'{e}zis-type constraint qualification conditions from \cite{combari1994sous} and is in general incomparable with \eqref{burke:qc:rewritten} (each being weaker than the other in some cases). In contrast, condition \eqref{cspwlqf} is weaker than the two others as expected from it only being sufficient when $f$ is a piecewise linear-quadratic function.

For all three conditions, we see that it will be useful to understand properties of the largest cone $K\in \mathcal{K}_F\cap\mathcal{K}_{g}$ for which we can rely on the following result.

\begin{lemma}[{\cite[Lemma 10, Proposition 24 and 55]{gissler2023note}}]\label{KFprop}
Let $F:\mathbb{R}^n\to\mathbb{R}^m\cup\{+\infty_\bullet\}$ be such that $\operatorname{dom}(F)$ is non-empty and convex. Then there exist a non-empty, closed and convex cone $K_F\subseteq \mathbb{R}^m$ such that $F$ is $K$-convex for a closed convex cone $K\subseteq \mathbb{R}^m$ if and only if $K_F\subseteq K$. 

Let $g\in \Gamma_0\left(\mathbb{R}^m\right)$. Then $g$ is $K$-increasing for a closed convex cone $K\subseteq \mathbb{R}^m$ if and only if $K\subseteq -\operatorname{hzn}(g)$, where 
\[\operatorname{hzn}(g):=\{d\in \mathbb{R}^m\,:\, g(x+d)\leq g(x)\quad \forall x\in \operatorname{dom}(g)\}\]
is the horizon cone of $g$.

In particular, we have that 
\[\mathcal{K}_F\cap\mathcal{K}_{g}=\left\{K\subseteq \mathbb{R}^m\,:\, 
    \begin{aligned}
    \text{$K$ is a closed convex cone and}\\
    \text{$K_F\subseteq K\subseteq -\operatorname{hzn}(g)$}\end{aligned}\right\}.\]
for any closed convex cone $K\subseteq \mathbb{R}^m$ we have that $F$ is $K$-convex and $g$ is $K$-increasing if and only if $K_F\subseteq K\subseteq -\operatorname{hzn}(g)$.
\end{lemma}

\subsection{Using K-convexity to satisfy necessary conditions}\label{grepairsection}

We illustrate the potential value of a cone $K\in \mathcal{K}_F\cap\mathcal{K}_{g,\operatorname{rng}(F)}\setminus \mathcal{K}_{g}$ through the following two examples.

\begin{example}[{\cite[Variant of Example 1]{burke2021study}}]\label{modifiedburkeexample1}
Let $F:\mathbb{R}^2\to \mathbb{R}^2$ be given by $F(x):=(0, x_2)$ and $g:\mathbb{R}^2\to \overline{\mathbb{R}}$ be given by 
\[g(w):=\begin{cases}\frac{1}{3} w_1^3 &w_1\geq 0\\+\infty &\text{o.w.}\end{cases}.\]

We have that $g\circ F\equiv 0$, $(g\circ F)^*=\delta_{\{(0,0)\}}$, $\operatorname{dom}(g)=\mathbb{R}_+\times \mathbb{R}$, $\operatorname{dom}(F)=\mathbb{R}^2$, $\operatorname{rge}(F)=\{0\}\times\mathbb{R}$ and
$f(x,u)=g(u)$ for all $x\in \mathbb{R}^2$ and $u\in \mathbb{R}^2$.
Then $f\in \Gamma_0\left(\mathbb{R}^2\times \mathbb{R}^2\right)$  and $0\in U=\operatorname{dom}(g)-\operatorname{rng}(F)=\mathbb{R}_+\times \mathbb{R}$, so by Theorem~\ref{simplecomposite} we have $(g\circ F)^*=\operatorname{cl}(\rho)$. Indeed, 
\[g^*(v):=\begin{cases}\frac{2}{3} v_1^{3/2} &v_1\geq 0,\; v_2=0\\
0&v_1<0,\;v_2=0\\+\infty &\text{o.w.}\end{cases}\]
and $(\mydot{y}{F})^*=\delta_{\{0\}\times\{y_2\}}$, so 
\begin{equation}
    \rho(v)=\inf_{y \in \mathbb{R}^2}\delta_{\{0\}\times\{y_2\}}(v)+g^*(y)=\delta_{\{0\}\times\mathbb{R}}(v)+g^*(0,v_2)=\delta_{\{(0,0)\}}(v).
\end{equation}
Hence,  $\rho$ is additionally closed and the infimum in its definition is attained. Unfortunately, $f$ is not piecewise linear-quadratic and $0\not\in \operatorname{rint}(U)=\operatorname{rint}(\mathbb{R}_+\times \mathbb{R})$ so condition \eqref{csgeneralf} is not satisfied and hence Theorem~\ref{simplecomposite} cannot be used to deduce these additional properties.

Now, we also have $K_F=\{(0,0)\}\subseteq -\operatorname{hzn}(g)=\{0\}\times\mathbb{R}$, so by Lemma~\ref{KFprop} the largest cone in $\mathcal{K}_F\cap\mathcal{K}_{g}$ is $K_1:=\{0\}\times\mathbb{R}$.  As expected,
\begin{equation}\label{modifiedburkeexample1:ricond}
    \operatorname{rint}\left(\operatorname{dom}(g)\right)\cap \operatorname{rint}\left(F\left(\operatorname{dom}(F)\right)+K\right)\neq \emptyset
\end{equation}
does not hold for $K=K_1$ as \eqref{csgeneralf} does not hold. However, \eqref{modifiedburkeexample1:ricond} does hold for $K=K_2:=\mathbb{R}_+\times \{0\}\in \mathcal{K}_F\cap\mathcal{K}_{g,\operatorname{rng}(F)}\setminus \mathcal{K}_{g}$.

\end{example}

\begin{example}[{\cite[Example 58]{gissler2023note}}]\label{gisslerexample1}
Let $F:\mathbb{R}^2\to \mathbb{R}^2$ be given by $F(x):=\left(\frac{1}{2}x_1^2, x_2\right)$ and $g:\mathbb{R}^2\to \overline{\mathbb{R}}$ be given by 
$g(w)=|w_1|$.

We have that  $g\circ F(x)=\frac{1}{2}w_1^2$, $(g\circ F)^*(v)=\frac{1}{2}v_1^2+\delta_{\{0\}}(v_2)$, $g^*=\delta_{[-1,1]\times \{0\}}$, 
\[(\mydot{y}{F})^*(v)=\begin{cases}\frac{1}{2y_1}v_1^2 &y_1>0, y_2=v_2\\+\infty &\text{o.w.}\end{cases},\]
so
\begin{equation*}
    \rho(v)=\inf_{\substack{y_1>0\\y_2\in \mathbb{R}}}\frac{1}{2y_1}v_1^2 +\delta_{\{y_2\}}(v_2)+g^*(y)=\frac{1}{2}v_1^2+g^*(1,v_2)=\frac{1}{2}v_1^2+\delta_{\{0\}}(v_2).
\end{equation*}
Hence, $(g\circ F)^*=\rho$, $\rho\in \Gamma_0(\mathbb{R}^2)$ and the infimum in the definition of $\rho$ is always attained. Unfortunately, 
$f(x,u)=g\left(F(x)+u\right)=\left|\frac{1}{2}x_1^2+u_1\right|$ is not convex and hence Theorem~\ref{simplecomposite} cannot be used to deduce these properties. 

Now, we can also check that $K_F=\mathbb{R}_+\times \{0\}$ and $\operatorname{hzn}(g)=\{0\}\times \mathbb{R}$ so by Lemma~\ref{KFprop} we have that $\mathcal{K}_F\cap\mathcal{K}_{g}$ is empty.  However, we also have $K_F\in \mathcal{K}_F\cap\mathcal{K}_{g,\operatorname{rng}(F)}\setminus \mathcal{K}_{g}$.

\end{example}

In both Examples~\ref{modifiedburkeexample1} and \ref{gisslerexample1} we have a cone $K\in \mathcal{K}_F\cap\mathcal{K}_{g,\operatorname{rng}(F)}\setminus \mathcal{K}_{g}$, which we would like to also be in $\mathcal{K}_F\cap\mathcal{K}_{g}$. Now, the motivation  for the definition of $\mathcal{K}_{g,\operatorname{rng}(F)}$ is that the behavior of $g$ outside of $\operatorname{rng}(F)$ does not affect $g\circ F$. This suggests that we may be able to modify $g$ to achieve $\mathcal{K}_{g,\operatorname{rng}(F)}=\mathcal{K}_{g}$ without changing $g\circ F$. Such modification is precisely what is provided by the following  observation from \cite[page 253]{pennanen1999graph}\footnote{Note that the definitions of $K$-convexity, $K$-epigraph, etc. in this reference are reversed compared to those in \cite{burke2021study,gissler2023note}. For instance, in \cite[page 1328]{burke2021study} it is noted that for $h:\mathbb{R}\to \overline{\mathbb{R}}$ convexity corresponds to $K$-convexity with $K=\mathbb{R}_+$, while in \cite[page 236]{pennanen1999graph} it is noted that for their version of it corresponds to $K$-convexity with $K=\mathbb{R}_-$.}

\begin{proposition}[\cite{pennanen1999graph}]\label{gK:lemma}
Let $g:\mathbb{R}^m\to \overline{\mathbb{R}}$, $F:\mathbb{R}^n\to \mathbb{R}^m\cup\{+\infty_\bullet\}$ and $K\in \mathcal{K}_F\cap\mathcal{K}_{g,\operatorname{rng}(F)}$. Then $g_K:\mathbb{R}^m\to \overline{\mathbb{R}}$ given by 
\begin{equation}\label{gkdef}
    g_K(w):=g \square \delta_{-K}=\inf_{k\in K}g(w+k)
\end{equation}is such that 
\begin{equation}\label{geqgk}
g(w)=g_K(w)\quad \forall w\in \operatorname{rge}(F).
\end{equation}
In addition, $K\in \mathcal{K}_F\cap\mathcal{K}_{g_K}$.

\end{proposition}

Replacing $g$ with $g_K$ for $K \in \mathcal{K}_F\cap\mathcal{K}_{g,\operatorname{rng}(F)}$  then has the potential to repair the issues in Examples~\ref{modifiedburkeexample1} and \ref{gisslerexample1}. However, while replacing $g$ with $g_K$ does not change $g\circ F$, it does change $g^*$ and hence $\rho$. Fortunately, using standard results (e.g. \cite[Theorem 11.23a]{rockafellar2009variational}) we get the following simple characterization of $g_K^*$ as a function of $g^*$.

\begin{lemma}\label{Kzerolemma}
    Let $K\subseteq \mathbb{R}^m$ be a closed convex cone, $g:\mathbb{R}^m\to \overline{\mathbb{R}}$ be a proper function and $g_K:\mathbb{R}^m\to \overline{\mathbb{R}}$ given by \eqref{gkdef}. Then $g^*_K=g^*+\delta_{-K^\circ}$.
\end{lemma}

Using Lemma~\ref{Kzerolemma} we see that the effect of replacing $g$ with $g_K$ in  definition \eqref{rhodef} of $\rho$ is replacing $y\in \mathbb{R}^m$ by $y\in -K^\circ$, which results precisely in definition \eqref{etadef} of $\eta$ and explains the difference between the definitions of $\rho$ and $\eta$\footnote{cf. \cite[Lemma 4]{burke2021study} and the $\psi^*$ in \cite[Lemma 5]{burke2021study} for which we have $f=\psi^*$.}.

Using Proposition~\ref{gK:lemma} and Lemma~\ref{Kzerolemma} we can resolve the issues in Examples~\ref{modifiedburkeexample1} and \ref{gisslerexample1} as follows.

\begin{example}[Example~\ref{modifiedburkeexample1} continued]
Using Proposition~\ref{gK:lemma} for  $K_2:=\mathbb{R}_+\times \{0\}\in \mathcal{K}_F\cap\mathcal{K}_{g,\operatorname{rng}(F)}$ we get $g_{K_2}:\mathbb{R}^2\to \overline{\mathbb{R}}$ given by 
\[g_{K_2}(w):=\begin{cases}\frac{1}{3} w_1^3 &w_1\geq 0\\0 &\text{o.w.}\end{cases},\]
for which $g_{K_2}\circ F=g\circ F$ and $\operatorname{dom}(g_{K_2})=\mathbb{R}^2$.  In addition, for $f_{K_2}:\mathbb{R}^2\times \mathbb{R}^2\to \overline{\mathbb{R}}$ given by $f_{K_2}(x,u)=g_{K_2}(F(x)+u)=g_{K_2}(u)$ we have $f_{K_2}\in \Gamma_0(\mathbb{R}^2\times \mathbb{R}^2)$, and $0\in \operatorname{rint}\left(U_{K_2}\right)=\operatorname{rint}\left(\operatorname{dom}(g_{K_2})-\operatorname{rng}(F)\right)=\mathbb{R}^2$. Hence, we can now apply Theorem~\ref{simplecomposite} to conclude that have $(g\circ F)^*=(g_{K_2}\circ F)^*=\rho_{K_2}$ where $\rho_{K_2}:\mathbb{R}^2\to \overline{\mathbb{R}}$ is given by 
\begin{equation*}
    \rho_{K_2}(v)=\inf_{y \in \mathbb{R}^2}\delta_{\{0\}\times\{y_2\}}(v)+g_{K_2}^*(y)=\delta_{\{0\}\times\mathbb{R}
    }(v)+g_{K_2}^*(0,v_2)=\delta_{\{(0,0)\}}(v),
\end{equation*}
where the last equation follows from Lemma~\ref{Kzerolemma}:
\[g_{K_2}^*(v)=g^*(v)+\delta_{-K_2^\circ}(v)=g^*(v)+\delta_{\mathbb{R}_+\times \mathbb{R}}(v)=\begin{cases}\frac{2}{3} v_1^{3/2} &v_1\geq 0,\; v_2=0\\+\infty &\text{o.w.}\end{cases}.\] 
\end{example}

\begin{example}[Example~\ref{gisslerexample1} continued]
Using Proposition~\ref{gK:lemma} for  $K_F=\mathbb{R}_+\times \{0\}\in \mathcal{K}_F\cap\mathcal{K}_{g,\operatorname{rng}(F)}$ we get $g_{K_F}:\mathbb{R}^2\to \overline{\mathbb{R}}$ given by 
$g_{K_F}(w):=\max\{w_1,0\}$ for which  $g_{K_F}\circ F=g\circ F$ and  $\operatorname{dom}(g_{K_F})=\mathbb{R}^2$.   In addition, for $f_{K_F}:\mathbb{R}^2\times \mathbb{R}^2\to \overline{\mathbb{R}}$ given by $f_{K_F}(x,u)=g_{K_F}(F(x)+u)=\max\left\{\frac{1}{2}x_1^2+u_1,0\right\}$ we have $f_{K_F}\in \Gamma_0(\mathbb{R}^2\times \mathbb{R}^2)$ and $0\in \operatorname{rint}\left(U_{K_F}\right)=\operatorname{rint}\left(\operatorname{dom}(g_{K_F})-\operatorname{rng}(F)\right)=\mathbb{R}^2$.  Hence, we can now apply Theorem~\ref{simplecomposite} to conclude that have $(g\circ F)^*=(g_{K_F}\circ F)^*=\rho_{K_F}$ where $\rho_{K_F}:\mathbb{R}^2\to \overline{\mathbb{R}}$ is given by 
\begin{equation*}
    \rho_{K_F}(v)=\inf_{\substack{y_1>0\\y_2\in \mathbb{R}}}\frac{1}{2y_1}v_1^2 +\delta_{\{y_2\}}(v_2)+g_{K_F}^*(y)=\frac{1}{2}v_1^2+g^*_{K_F}(1,v_2)=\frac{1}{2}v_1^2+\delta_{\{0\}}(v_2).
\end{equation*}
where the last equation follows from Lemma~\ref{Kzerolemma}:
\[g_{K_F}^*=g^*+\delta_{-K_F^\circ}=\delta_{[-1,1]\times\{0\}}+\delta_{\mathbb{R}_+\times\{0\}}=\delta_{[0,1]\times \{0\}}.\]
\end{example}

We end this section by noting that while $g_K$ preserves convexity of $g$ it does not necessarily preserve $g$ being proper or closed. Hence, repairing a bad choice of $g$ to describe $g\circ F$ may require additional steps such as first restricting $g$ to $\operatorname{rng}(F)$ (if this set is convex) or $\operatorname{rng}(F)+K$ (which is convex if $F$ is $K$-convex). This is illustrated in the following example. 

\begin{example}\label{sqrtexample}
Let $K=\mathbb{R}_+\times\{0\}$ and $g:\mathbb{R}^2\to \overline{\mathbb{R}}$ be given by
    \[g(w):=\begin{cases}
    -\sqrt{w_1 w_2} &w_1\geq 0,w_2\geq 0\\
    +\infty&\text{o.w.}
    \end{cases}.\]
Assume we want to compose $g$ with a $K$-convex function $F:\mathbb{R}\to\mathbb{R}^2\cup\{+\infty_\bullet\}$ for which $\operatorname{rng}(F)=\mathbb{R}_+\times\{0\}$. Using Proposition~\ref{gK:lemma} for $K\in \mathcal{K}_F\cap\mathcal{K}_{g,\operatorname{rng}(F)}\setminus \mathcal{K}_{g}$ we get $g_{K}:\mathbb{R}^2\to \overline{\mathbb{R}}$ given by
    \[g_K(w)=\begin{cases}
    -\infty&w_2>0\\
    0&w_2=0\\
    +\infty&\text{o.w.}
\end{cases},\]
which is not proper. However, if we first restrict $g$ to $\operatorname{rge}(F)=\operatorname{rge}(F)+K=\mathbb{R}_+\times\{0\}$ we get $(g+\delta_{\operatorname{rge}(F)})_{K}:\mathbb{R}^2\to \overline{\mathbb{R}}$ given by
\[(g+\delta_{\operatorname{rge}(F)})_{K}(w)=\begin{cases}
    0&w_2=0\\
    +\infty&\text{o.w.}
\end{cases},\]
for which $(g+\delta_{\operatorname{rge}(F)})_{K}\in\Gamma_0(\mathbb{R}^2)$. 
\end{example}

\subsection{Additive composite functions}\label{additive:comparison:section}

Using standard results on the conjugate of the sum of two functions (e.g. \cite[Theorem 16.4]{rockafellar2015convex}), \cite{burke2021study} also extends \cref{burketheo} to the case $f_0\not\equiv 0$ through \cite[Corollary 3]{burke2021study}, which we can restate as follows.

\begin{corollary}[{\cite[Corollary 3]{burke2021study}}]\label{burkecorollary}
Let $K\subseteq \mathbb{R}^m$ be a closed convex cone, $f_0:\mathbb{R}^n\to \overline{\mathbb{R}}$, $g:\mathbb{R}^m\to \overline{\mathbb{R}}$, $F:\mathbb{R}^n\to \mathbb{R}^m\cup\{+\infty_\bullet\}$, $\mydot{y}{F}:\mathbb{R}^n\to \overline{\mathbb{R}}$ be given by \eqref{extendedscalarizationdef}, and $f_0 + g \circ F:\mathbb{R}^n\to \overline{\mathbb{R}}$ be given by \eqref{generalcompositef}.

If $g\in\Gamma_0(\mathbb{R}^m)$,   $y_0 f_0+
    \mydot{y}{F}\in \Gamma_0(\mathbb{R}^n)$   for all $(y_0,y)\in\mathbb{R}_+\times (-K^\circ)$, $K\in \mathcal{K}_F\cap\mathcal{K}_{g}$, and
\begin{equation}\label{burkecorollary:qc}
  \operatorname{rint}\left(\operatorname{dom}(g)-K\right)\cap  F\left(\operatorname{rint}\left(\operatorname{dom}(f_0)\right)\cap\operatorname{rint}\left(\operatorname{dom}(F)\right)\right)\neq \emptyset,
\end{equation}
then, for any $\bar{v}\in \mathbb{R}^n$ we have
\begin{equation}
     \left(f_0+g \circ F\right)^*(\bar{v})=\min_{y \in-K^\circ,\;w\in \mathbb{R}^n}f_0^*(w)+\left(\mydot{y}{ F}\right)^*(\bar{v}-w)+g^*(y).
\end{equation}
\end{corollary}
From \cref{simplecompositeadditive} we can get the following analogous corollary

\begin{corollary}\label{simplecompositecoro}
For $f_0:\mathbb{R}^n\to \overline{\mathbb{R}}$, $F:\mathbb{R}^n\to \mathbb{R}^m\cup\{+\infty_\bullet\}$ and $g:\mathbb{R}^m\to \overline{\mathbb{R}}$, let $\mydot{y}{F}:\mathbb{R}^n\to \overline{\mathbb{R}}$ given by \eqref{extendedscalarizationdef}, 
  $f_0 + g \circ F:\mathbb{R}^n\to \overline{\mathbb{R}}$ be given by \eqref{generalcompositef} and $f:\mathbb{R}^n\times \mathbb{R}^m\to \overline{\mathbb{R}}$ be given by \eqref{extendedcompositedef}.

If $f_0\in \Gamma(\mathbb{R}^n)$,  $h\in \Gamma\left(\mathbb{R}^n\times \mathbb{R}^m\right)$, $f\in \Gamma_0(\mathbb{R}^n\times \mathbb{R}^m)$, 
\begin{equation}\label{simplecompositecoro:inf:conv:condion}
    \operatorname{rint}\left(\operatorname{dom}(f_0)\right)\cap\operatorname{rint}\left(\operatorname{dom}(F)\right)\neq \emptyset
\end{equation}
and there exists $K\in\mathcal{K}_F\cap\mathcal{K}_{g}$ such that either 
  \begin{equation}\label{simplecompositecoro:general:condition}
   \operatorname{rint}\left(\operatorname{dom}(g)\right)\cap \operatorname{rint}\left(F\left(\operatorname{dom}(f_0)\cap\operatorname{dom}(F)\right)+K\right)\neq \emptyset,
\end{equation}
or  $f$ is PWLQ, and
 \begin{equation}\label{simplecompositecoro:pwlq:condition}
     \left(\operatorname{dom}(g)-K\right)\cap \left(F\left(\operatorname{dom}(f_0)\cap\operatorname{dom}(F)\right)\right)\neq \emptyset,
\end{equation}
then, for any $\bar{v}\in \mathbb{R}^n$ we have
\begin{equation}
    \left(f_0 + g \circ F\right)^*(\bar{v})=\min_{y \in \mathbb{R}^m,\;w\in \mathbb{R}^n}f_0^*(w)+\left(\mydot{y}{ F}\right)^*(\bar{v}-w)+g^*(y),
\end{equation}
\end{corollary}

Similarly to the case $f_0\equiv 0$ we can see that PWLQ conditions \eqref{simplecompositecoro:inf:conv:condion} and \eqref{simplecompositecoro:pwlq:condition} from \cref{simplecompositecoro} are weaker than condition \eqref{burkecorollary:qc} from \cref{burkecorollary}, while general conditions \eqref{simplecompositecoro:inf:conv:condion} and \eqref{simplecompositecoro:general:condition} from \cref{simplecompositecoro} and \eqref{burkecorollary:qc} from \cref{burkecorollary} are incomparable.

Finally, note that for \cref{simplecompositecoro} we can relax \eqref{simplecompositecoro:inf:conv:condion} if we have additional information on $f_0$ and $F$. For instance, based on the comments after \cite[Theorem E.2.3.2]{hiriart2004fundamentals} we can relax \eqref{simplecompositecoro:inf:conv:condion} to $\operatorname{dom}(f_0)\cap\operatorname{dom}(F)\neq \emptyset$ if $f_0$  and $\mydot{\bar{y}}{ F}$ are polyhedral for all $y\in \mathbb{R}^m$.

\section{Omitted Proofs}\label{omitted:proofs:section}

\subsection{Proofs from Section~\ref{convex:comp:section}}
\GNLPLagrangianLemma*
\begin{proof}{\textbf{Proof of Lemma~\ref{gnlplagrangianlemma}}\hspace*{0.5em}}
\label{gnlplagrangianlemma:proof}
If $x \in \operatorname{dom}(F)$, then 
\begin{align*}
l(x,y)= -\sup_{u\in\mathbb{R}^m}\left\{\mydot{y}{u}-f(x,u)\right\}&=f_0(x)-\sup_{u\in\mathbb{R}^m}\left\{\mydot{y}{u}-g\left(F(x)+u\right)\right\}\\
&=f_0(x)-\sup_{u\in\mathbb{R}^m}\left\{\mydot{y}{u-F(x)}-g\left(u\right)\right\}\\
&= f_0(x)+\mydot{y}{F(x)} -g^*(y).
\end{align*}
If $x \not\in \operatorname{dom}(F)$, then both sides of \eqref{gnlplagrangianlemma:char:a} are $+\infty$.

For \eqref{gnlplagrangianlemma:char:b} we have
\begin{align*}
f^*(v,y)&=\sup_{x\in \mathbb{R}^n,u\in \mathbb{R}^m}\left\{\mydot{v}{x}+\mydot{y}{u}-f(x,u)\right\}\\
&=\sup_{x \in \operatorname{dom}(F),u\in \mathbb{R}^m}\left\{\mydot{v}{x}+\mydot{y}{u}-f_0(x) -g(F(x)+u)\right\}\\
&=\sup_{x \in \operatorname{dom}(F)}\left\{\mydot{v}{x}-f_0(x)-\inf_{u\in \mathbb{R}^m}\left\{g(F(x)+u)-\mydot{y}{u}\right\}\right\}\\
&=\sup_{x \in \operatorname{dom}(F)}\left\{\mydot{v}{x}-f_0(x)-\inf_{u\in \mathbb{R}^m}\left\{g(u)-\mydot{y}{u-F(x)}\right\}\right\}\\
&=\sup_{x \in \operatorname{dom}(F)}\left\{\mydot{v}{x}-f_0(x)-\mydot{y}{F(x)}+g^*(y)\right\}\\
& = \left(f_0 +\mydot{y}{ F}\right)^*(v)+g^*(y)
\end{align*}
 \Halmos \end{proof}

\subsection{Proofs from Section~\ref{connection:section}}

\allconvexityprop*
\begin{proof}{\textbf{Proof}\hspace*{0.5em}}
\label{allconvexityprop:proof}
First, \eqref{weakallconvexitypropequiv} follows from \cite[Proposition 1.b]{burke2021study}. For \eqref{allconvexitypropequiv} note that if $F$ is $K$-convex, then $H:\mathbb{R}^n \times \mathbb{R}^m\to \mathbb{R}^m\cup\{+\infty_\bullet\}$ given by $H(x,u)=F(x)+u$ is also $K$-convex so $\mathcal{K}_F\subseteq \mathcal{K}_H$. In addition, $\operatorname{rng}(H)=\mathbb{R}^m$ so $\mathcal{K}_{g}=\mathcal{K}_{g,\operatorname{rng}(H)}$. Hence, $\mathcal{K}_F\cap\mathcal{K}_{g}\neq \emptyset$ implies $\mathcal{K}_H\cap\mathcal{K}_{g,\operatorname{rng}(H)}\neq \emptyset$, which by \eqref{weakallconvexitypropequiv} implies that $g\circ H = f$ is convex.

Under the additional assumptions on $g$ and $F$ we have that $f$ is proper and for any $\bar{x}\in \mathbb{R}^n$ we have that $f(\bar{x},\cdot)$ is closed and convex (possibly with $f(\bar{x},\cdot)\equiv +\infty$). Hence, $f$ is a \emph{dualizing parametrization} as in \cite[Definition 11.45]{rockafellar2009variational}. Then, by \cite[Proposition 11.48]{rockafellar2009variational} and Lemma~\ref{gnlplagrangianlemma} (with $f_0\equiv 0$) we have that $l:\mathbb{R}^n\times \mathbb{R}^m\to \overline{\mathbb{R}}$ given by
$l(x,y)= \mydot{y}{F(x)} -g^*(y)$ is convex in $x$ if and only if $f(x,u)$ is convex in $(x,u)$. Hence, 
\begin{subequations}
\begin{align}
  \text{$f(x,u)$ is convex in $(x,u)$} &\Leftrightarrow \mydot{y}{F} \text{ is convex} \quad\forall y\in \operatorname{dom}\left(g^*\right)\\  
  \label{jointconvexityhznconvexlemma:eq1}
    &\Leftrightarrow \mydot{y}{F} \text{ is convex} \quad\forall y\in \operatorname{cl}\operatorname{cone}\operatorname{dom}\left(g^*\right)\\  
    &\Leftrightarrow \mydot{y}{F} \text{ is convex} \quad\forall y\in \left(\operatorname{hzn}(g)\right)^\circ,  \label{jointconvexityhznconvexlemma:eq2}\\  
    &\Leftrightarrow F \text{ is $(-\operatorname{hzn}(g))$-convex}, \label{jointconvexityhznconvexlemma:eq3}
\end{align}
\end{subequations}
where \eqref{jointconvexityhznconvexlemma:eq1} follows because the set of convex functions is a closed convex cone, \eqref{jointconvexityhznconvexlemma:eq2} follows from \cite[Theorem 14.2]{rockafellar2015convex}, and \eqref{jointconvexityhznconvexlemma:eq3} follows from \cite[Theorem 3]{burke2021study} by noting that $\operatorname{hzn}(g)$ is a closed convex cone  (e.g \cite[Proposition 55]{gissler2023note}). Then, by Lemma~\ref{KFprop} we gave that $(-\operatorname{hzn}(g))\in \mathcal{K}_F\cap\mathcal{K}_{g}$.
 \Halmos \end{proof}

\firstequivlemmastatement*
\begin{proof}{\textbf{Proof of Lemma~\ref{firstequivlemma}\hspace*{0.5em}}}
\label{firstequivlemma:proof}
Follows from the proof of 
\cite[Corollary 1]{burke2021study} as follows. 

Because $\operatorname{rge}(F)\neq \emptyset$ we have $K\text{-}\operatorname{epi}(F)\neq\emptyset$ and \cite[Lemma 5]{burke2021study} implies $f(x,u)= \psi^*(x,u)$ for $\psi:\mathbb{R}^n\times \mathbb{R}^m\to \overline{\mathbb{R}}$ given by $\psi(v,y):=\sigma_{K\text{-}\operatorname{epi}(F)}(v,-y)+g^*(y)$. Under the assumptions,  $\sigma_{K\text{-}\operatorname{epi}(F)}\in \Gamma_0\left(\mathbb{R}^n\times \mathbb{R}^m\right)$ and $g^*\in \Gamma_0\left(\mathbb{R}^m\right)$. Then, either $\psi\equiv +\infty$ or $\psi\in \Gamma_0\left(\mathbb{R}^n\times \mathbb{R}^m\right)$. However, by \cite[Exercise 11.2]{rockafellar2009variational}, if $\psi\equiv +\infty$, then $f\equiv-\infty$, which contradicts $g\in \Gamma_0\left(\mathbb{R}^m\right)$ being proper. Hence, $\psi\in\Gamma_0\left(\mathbb{R}^n\times \mathbb{R}^m\right)$ and $f\in \Gamma_0\left(\mathbb{R}^n\times \mathbb{R}^m\right)$.

 \Halmos \end{proof}

\cscomparisonlemma*
\begin{proof}{\textbf{Proof of Lemma~\ref{cscomparisonlemma:label}}\hspace*{0.5em}}
\label{cscomparisonlemma:proof}
For $f_0\equiv 0$ we have that $U=\operatorname{dom}(g)-F\left( \operatorname{dom}(F)\right)$. In addition, by \cite[Lemma 3]{burke2021study} if $K\in \mathcal{K}_g$, then $\operatorname{dom}(g)=\operatorname{dom}(g)-K$. 
Hence, 
\[U=\left(\operatorname{dom}(g)-K\right)-F\left( \operatorname{dom}(F)\right)=\operatorname{dom}(g)-\left(F\left( \operatorname{dom}(F)\right)+K\right),\] so $0\in U$ is equivalent to \eqref{cspwlqf}. We also have that if $K\in \mathcal{K}_F$, then $\operatorname{rng}(F)+K$ is convex because it is the projection of the $K$-epigraph of $F$ onto the $z$ variables. Then, by convexity of $\operatorname{dom}(g)$ and \cite[Corollary 6.6.2]{rockafellar2015convex} we have that 
\[\operatorname{rint}\left(U\right)=\operatorname{rint}\left(\operatorname{dom}(g)\right)-\operatorname{rint}\left(F\left( \operatorname{dom}(F)\right)+K\right)\]
and hence $0\in \operatorname{rint}(U)$ is equivalent to \eqref{csgeneralf}.

Condition \eqref{burke:qc:rewritten} is simply a re-ordering of \eqref{burke:qc}.
 \Halmos \end{proof}

\bibliographystyle{plain}
\bibliography{references}

%
\begin{APPENDICES}
\section{Proofs and Examples from Section~\ref{perturbationsection}}\label{perturbationsection:proofs}
\subsection{Proof of Theorem~\ref{newGWDOP}, Proposition~\ref{partialconvex:label} and Theorem~\ref{NewConvexStrongDual}}

To prove Theorem~\ref{newGWDOP}, Proposition~\ref{partialconvex:label} and Theorem~\ref{NewConvexStrongDual} we rely on the following lemmas.

\begin{lemma}\label{doubleconjugatelemma}
Let $h:\mathbb{R}^n\to \overline{R}$. Then $h^*$ is convex and 
\begin{equation}\label{doubleconjugateineq}
    h\geq \operatorname{co}(h)\geq \operatorname{cl}\operatorname{co}(h) \geq h^{**}.
\end{equation}
Equality holds throughout \eqref{doubleconjugateineq} if $h\in \Gamma_0\left(\mathbb{R}^m\right)$ or  $h\equiv +\infty$.
If $h\in \Gamma\left(\mathbb{R}^m\right)$, then both $h^*\in \Gamma_0\left(\mathbb{R}^m\right)$, $h^{**}\in \Gamma_0\left(\mathbb{R}^m\right)$ and $\operatorname{cl}(h) = h^{**}$.
\end{lemma}
\begin{proof}{\textbf{Proof}\hspace*{0.5em}}
Direct from \cite[Theorem 11.1 and Exercise 11.2]{rockafellar2009variational}.
 \Halmos \end{proof}

\begin{lemma}\label{pqlemma}
For $f:\mathbb{R}^n\times \mathbb{R}^m\to \overline{\mathbb{R}}$, $\bar{v}\in \mathbb{R}^n$, and $\bar{u}\in \mathbb{R}^m$ let $p_{\bar{v}}:\mathbb{R}^m\to \overline{R}$ and $q_{\bar{u}}:\mathbb{R}^n\to\overline{R}$ be given by
\begin{align*}
    p_{\bar{v}}(\bar{u})&:=\inf_{x\in \mathbb{R}^n}\left\{ f(x,\bar{u})-\mydot{\bar{v}}{x}\right\}\\
    q_{\bar{u}}(\bar{v})&:=\inf_{y\in\mathbb{R}^m}\left\{ f^*(\bar{v},y) - \mydot{y}{\bar{u}}\right\}.
\end{align*}
Then, $q_{\bar{u}}$ is convex and for all $\bar{x}\in \mathbb{R}^n$, $\bar{y}\in\mathbb{R}^m$, $\bar{v}\in \mathbb{R}^n$, and $\bar{u}\in \mathbb{R}^m$ we have
\begin{subequations}
\begin{align}
p_{\bar{v}}^*(\bar{y})&=f^*(\bar{v},\bar{y})\label{pqlemmaeq:1}\\
    p_{\bar{v}}^{**}(\bar{u})&=\sup_{y\in \mathbb{R}^m}\left\{\mydot{\bar{u}}{y}-f^*(\bar{v},y)\right\}=-q_{\bar{u}}(\bar{v})\label{pqlemmaeq:2}\\
    q_{\bar{u}}^*(\bar{x})&=f^{**}(\bar{x},\bar{u})\leq f(\bar{x},\bar{u})\label{pqlemmaeq:3}\\
    q_{\bar{u}}^{**}(\bar{v})&\geq\sup_{x\in \mathbb{R}^n} \left\{\mydot{\bar{v}}{x} -f(x,\bar{u})\right\}=-p_{\bar{v}}(\bar{u}).\label{pqlemmaeq:4}
\end{align}
\end{subequations}
If $f\in \Gamma_0\left(\mathbb{R}^n\times \mathbb{R}^m\right)$, then $p_{\bar{v}}$ is also convex and  equality holds in \eqref{pqlemmaeq:3} and \eqref{pqlemmaeq:4}.
\end{lemma}
\begin{proof}{\textbf{Proof}\hspace*{0.5em}}
Convexity of $q_{\bar{u}}$ follows from \cite[Proposition 2.22]{rockafellar2009variational} by noting that by Lemma~\ref{doubleconjugatelemma} $f^*$ is convex. 

The equations in \eqref{pqlemmaeq:1} and \eqref{pqlemmaeq:3} follow by expanding the definitions of $p_{\bar{v}}^*$ and $q_{\bar{u}}^*$ and the inequality in \eqref{pqlemmaeq:3} follows from Lemma~\ref{doubleconjugatelemma}. We then directly get \eqref{pqlemmaeq:2} and \eqref{pqlemmaeq:4}.

Under the additional conditions on $f$ convexity of $p_{\bar{v}}$ follows from \cite[Proposition 2.22]{rockafellar2009variational} and the additional equalities follow from  Lemma~\ref{doubleconjugatelemma} applied to $f$.
 \Halmos \end{proof}

\begin{lemma}\label{subdifferentiallemma}
Let $h:\mathbb{R}^n\to \overline{\mathbb{R}}$ be a convex function, $\bar{x}\in \mathbb{R}^n$ and $\bar{v}\in \mathbb{R}^n$. Then, 
\begin{align*}
    \arg \max_{x\in \mathbb{R}^n} \left\{\mydot{\bar{v}}{x} - h(x)\right\}=\partial h^*(\bar{v})\\
    \arg \max_{v\in \mathbb{R}^n} \left\{\mydot{{v}}{\bar{x}} - h^*(v)\right\}=\partial h(\bar{x})
\end{align*}
\end{lemma}
\begin{proof}{\textbf{Proof}\hspace*{0.5em}}
Follows from \cite[Theorem 11.3]{rockafellar2009variational} and  \cite[Exercise 11.2]{rockafellar2009variational}.
 \Halmos \end{proof}

\GeneralWeakDual*
\begin{proof}{\textbf{Proof of Theorem~\ref{newGWDOP}}\hspace*{0.5em}}
\label{newdualweakGNLPduality:proof}

We begin by noting that for all $\bar{v}\in \mathbb{R}^n$ and $\bar{y}\in \mathbb{R}^m$ we have \begin{equation}\label{newGNLP:proof:eq}
    f^*(\bar{v},\bar{y})=\sup_{x\in\mathbb{R}^n}\left\{\mydot{\bar{v}}{x}-l(x,\bar{y})\right\}
\end{equation} which yields the last equation in \eqref{newweakGNLPduality} and \eqref{newdualweakGNLPduality}.

Then \eqref{newweakGNLPduality} and \eqref{newdualweakGNLPduality} follow from Lemma~\ref{pqlemma} and Lemma~\ref{doubleconjugatelemma} (applied to the functions $p_{\bar{v}}$ and $q_{\bar{u}}$ defined in Lemma~\ref{pqlemma}).

The equivalence between \eqref{newGNLP:optcond1}, \eqref{newGNLP:optcond2} follows by noting that for any $(\bar{x},\bar{y})\in \mathbb{R}^n\times \mathbb{R}^m$
\begin{subequations}\label{sadelineqs}
\begin{align} f(\bar{x},\bar{u})-\mydot{\bar{v}}{\bar{x}}&\geq \left(f(\bar{x},\cdot)\right)^{**}(\bar{u})-\mydot{\bar{v}}{\bar{x}}\\&=\sup_{y\in \mathbb{R}^{m}}\left\{l(\bar{x},y)+\mydot{\bar{u}}{ y}-\mydot{\bar{v}}{\bar{x}}\right\}\\
&\geq l(\bar{x},\bar{y})+\mydot{\bar{u}}{ \bar{y}}-\mydot{\bar{v}}{\bar{x}}\\
&\geq \inf_{x\in \mathbb{R}^n}\left\{ l({x},\bar{y}) +\mydot{\bar{u}}{\bar{y}}-\mydot{\bar{v}}{{x}}\right\}\\
&= \mydot{\bar{u}}{\bar{y}}-f^*(\bar{v},\bar{y})\end{align}
\end{subequations}
where the first inequality follows from Lemma~\ref{doubleconjugatelemma}, the first equality follows from $l(x,y)= -\left(f(x,\cdot)\right)^*(y)$ and the last equality follows from \eqref{newGNLP:proof:eq}.
 \Halmos \end{proof}

\partialconvex*
\begin{proof}{\textbf{Proof of Proposition~\ref{partialconvex:label}}\hspace*{0.5em}}
\label{partialconvex:proof}

All results follow from \eqref{sadelineqs} by noting that if $f$ is proper and $f(x,u)$ is closed and convex in $u$, then $f(x,\cdot)\equiv+\infty$ or $f(x,u)$ is proper, closed and convex in $u$ so by Lemma~\ref{doubleconjugatelemma} we have $f(\bar{x},\bar{u})= \left(f(\bar{x},\cdot)\right)^{**}(\bar{u})$.
 \Halmos \end{proof}

\CSD*
\begin{proof}{\textbf{Proof of Theorem~\ref{NewConvexStrongDual}}\hspace*{0.5em}}
\label{NewConvexStrongDual:proof}

Convexity of $q_{\bar{u}}$ follows from Lemma~\ref{pqlemma}.

The convexity statements for $l$ follow from \cite[Proposition 11.48]{rockafellar2009variational} and the equivalence between \eqref{newGNLP:optcond3} and \eqref{newGNLP:optcond4} follows from Lemma~\ref{subdifferentiallemma}.

The statements on $p_{\bar{v}}$ and $q_{\bar{u}}$ follow from Lemmas~\ref{pqlemma} and \ref{subdifferentiallemma}.

The remainder of the results follow from Lemma~\ref{pqlemma}, Lemma~\ref{doubleconjugatelemma} (applied to the functions $p_{\bar{v}}$ and $q_{\bar{u}}$).
 \Halmos \end{proof}

\subsection{Qualification conditions}\label{perturbationsection:qualification}

Most sufficient conditions for the qualification conditions in Theorem~\ref{NewConvexStrongDual} are based on the domains of functions $p_{\bar{v}}$ and $q_{\bar{u}}$. The following straightforward lemma notes that for such conditions we can restrict our attention to the cases $\bar{v}=$ and $\bar{u}=0$.

\begin{lemma}\label{pqdomlemmalabel}
Let  $f:\mathbb{R}^n\times \mathbb{R}^m\to \overline{\mathbb{R}}$, $\bar{v}\in \mathbb{R}^n$, $\bar{u}\in \mathbb{R}^m$, and $p_{\bar{v}}:\mathbb{R}^m\to \overline{\mathbb{R}}$, $q_{\bar{u}}:\mathbb{R}^n\to\overline{\mathbb{R}}$ be the functions given by \eqref{pqdef}. Then, 
\begin{alignat*}{3}
     \operatorname{dom}(p_0) &= \operatorname{dom}(p_{\bar{v}})&\quad&\forall \bar{v}\in \mathbb{R}^n\\
      \operatorname{dom}(q_0) &= \operatorname{dom}(q_{\bar{u}})&\quad&\forall \bar{u}\in \mathbb{R}^m.
\end{alignat*}
\end{lemma}

\subsubsection{Conditions for asymptotic strong duality}

For $f\in \Gamma\left(\mathbb{R}^n\times \mathbb{R}^m\right)$, both $p_{\bar{v}}$ and $q_{\bar{u}}$ are guaranteed to not be identically $+\infty$. Exploiting some dual relationships between these functions (e.g. Lemma~\ref{pqlemma}), we can ensure one of the function avoids taking the value $-\infty$ through conditions on the domain of the other one as follows.

\begin{restatable}{lemma}{pqproperlemma}\label{pqproperlemmalabel}
Let  $f\in \Gamma\left(\mathbb{R}^n\times \mathbb{R}^m\right)$, $\bar{v}\in \mathbb{R}^n$, $\bar{u}\in \mathbb{R}^m$, and $p_{\bar{v}}:\mathbb{R}^m\to \overline{\mathbb{R}}$, $q_{\bar{u}}:\mathbb{R}^n\to\overline{\mathbb{R}}$ be the functions given by \eqref{pqdef}.

\begin{itemize}
    \item If $\bar{v}\in \operatorname{dom}\left(q_{0}\right)$ (or equivalently if $f^*(\bar{v},\cdot)\not\equiv +\infty$), then $p_{\bar{v}}$ is proper, and
      \item if $\bar{u}\in \operatorname{dom}\left(p_{0}\right)$  (or equivalently if $f(\cdot,\bar{u})\not\equiv +\infty$), then $q_{\bar{u}}$ is proper
\end{itemize}
\end{restatable}
\begin{proof}{\textbf{Proof}\hspace*{0.5em}}
We start by noting that $ p_{\bar{v}}\not\equiv +\infty$ for all $\bar{v}\in \mathbb{R}^n$ and $q_{\bar{u}}\not\equiv +\infty$ for all $\bar{u}\in \mathbb{R}^m$. This follows from $f$ and $f^*$ being proper under the assumptions on $f$ and Lemma~\ref{doubleconjugatelemma}.

For the remaining statements, first note that if $f^*(\bar{v},\cdot)\not\equiv +\infty$, then $q_{\bar{u}}(\bar{v})<+\infty$ for all $\bar{u}\in \mathbb{R}^m$. Similarly, if $f(\cdot,\bar{u})\not\equiv +\infty$, then $p_{\bar{v}}(\bar{u})<+\infty $ for all $\bar{v}\in \mathbb{R}^n$. In addition,  by Lemmas~\ref{doubleconjugatelemma} and \ref{pqlemma}, we have that   $p_{\bar{v}}(\bar{u})\geq -q_{\bar{u}}(\bar{v})$ and $q_{\bar{u}}(\bar{v})\geq -p_{\bar{v}}(\bar{u})$ for all $\bar{v}\in \mathbb{R}^n$ and $\bar{u}\in \mathbb{R}^m$. Hence, if $f^*(\bar{v},\cdot)\not\equiv +\infty$, then  $p_{\bar{v}}(\bar{u})>-\infty$ for all $\bar{u}\in\mathbb{R}^m$ and if $f(\cdot,\bar{u})\not\equiv +\infty$, then $q_{\bar{u}}(\bar{v})>-\infty$ for all $\bar{v}\in \mathbb{R}^n$. The result follows because we already have that $ p_{\bar{v}}$ and $q_{\bar{u}}$  are never identically $+\infty$.
 \Halmos \end{proof}

\subsubsection{Conditions for non-asymptotic strong duality}

A simple Slater-like condition for non-asymptotic strong duality can be obtained by replacing $\operatorname{dom}$ by the interior of $\operatorname{dom}$ in Lemma~\ref{pqproperlemmalabel}. However, Proposition~\ref{pqclosedproposition:label} below shows we can obtain slightly weaker conditions in general and much weaker conditions for PWLQ functions. 

\begin{restatable}{proposition}{pqclosedproposition}\label{pqclosedproposition:label}
Let $f\in \Gamma_0\left(\mathbb{R}^n\times \mathbb{R}^m\right)$, $\bar{v}\in \mathbb{R}^n$, $\bar{u}\in \mathbb{R}^m$, and $p_{\bar{v}}:\mathbb{R}^m\to \overline{\mathbb{R}}$, $q_{\bar{u}}:\mathbb{R}^n\to\overline{\mathbb{R}}$ be the functions given by \eqref{pqdef}.

Then $q_{\bar{u}}\in \Gamma_0\left(\mathbb{R}^n\right)$  if any of the following holds
\begin{itemize}
\item $\bar{u}\in \operatorname{rint} \operatorname{dom}(p_0)$, or
\item $\bar{u}\in  \operatorname{dom}(p_0)$ and $f$ is PWLQ.
\end{itemize}
\end{restatable}
\begin{proof}{\textbf{Proof}\hspace*{0.5em}}
First note that
\begin{equation}\label{pqclosedproposition:Lepieq}
    L(C)\subseteq \operatorname{epi}(q_{\bar{u}})\subseteq \operatorname{cl}(L(C))
\end{equation}
for $C:=\left\{(v,y,z)\in \mathbb{R}^{n+m+1}\,:\, f^*(v,y)-\mydot{\bar{u}}{y}\leq z\right\}$
and the linear transformation $L:\mathbb{R}^{n+m+1}\to \mathbb{R}^{n+1}$ given by $L(v,y,z)=(v,z)$ whose nullspace is $N(L):=\left\{(v,y,z)\in \mathbb{R}^{n+m+1}\,:\, v=0,\;z=0\right\}$ (e.g. \cite[Proposition 3.3.1]{bertsekas2009convex}). In addition, by Lemma~\ref{doubleconjugatelemma} we have that $f^*\in \Gamma_0\left(\mathbb{R}^n\times \mathbb{R}^m\right)$, so $C$ is closed and convex. By \cite[Exercise 3.29]{rockafellar2009variational} we have that 
\[C_\infty\cap N(L)=\left\{(v,y,z)\in \mathbb{R}^{n+m+1}\,:\, f^{*\infty}(0,y)-\mydot{\bar{u}}{y}\leq 0,\; v=0,\;z=0\right\}.\]

We claim that 
\begin{equation}\label{pqclosedproposition:eq1}
    C_\infty\cap N(L)\subseteq(-(C_\infty\cap N(L)))\cap (C_\infty\cap N(L))
\end{equation}
is equivalent to $\bar{u}\in \operatorname{rint} \operatorname{dom}(p_0)$. To show this first note that $\bar{u}\in \operatorname{rint} \operatorname{dom}(p_0)$ is equivalent to having  $y\neq 0$ imply
\begin{align*}
    \sigma_{\operatorname{dom}(p_0)}(y)&> \mydot{\bar{u}}{y},\quad\text{ or}\quad
    -\sigma_{\operatorname{dom}(p_0)}(-y)=\sigma_{\operatorname{dom}(p_0)}(y)=\mydot{\bar{u}}{y}.
\end{align*}
In addition, by \cite[Theorem 11.5]{rockafellar2009variational}
$\sigma_{\operatorname{dom}(p_0)}(y)=\sigma_{\operatorname{dom}(f)}(0,y)=f^{*\infty}(0,y)$. Hence, $\bar{u}\in \operatorname{rint} \operatorname{dom}(p_0)$ is also equivalent to having  $y\neq 0$ imply
\begin{align*}
    f^{*\infty}(0,y)-\mydot{\bar{u}}{y}&> 0,\quad\text{ or}\quad
        f^{*\infty}(0,y)-\mydot{\bar{u}}{y}=f^\infty(0,-y)-\mydot{\bar{u}}{-y}=0,
\end{align*}
which in turn is equivalent to \eqref{pqclosedproposition:eq1}

Now by \cite[Proposition 1.4.13]{bertsekas2009convex} we have that $L(C)$ is closed if $C_\infty\cap N(L)\subseteq(-(C_\infty\cap N(L)))\cap (C_\infty\cap N(L))\subseteq (-C_\infty)\cap C_\infty$. Hence, the first sufficient condition implies $L(C)$ is closed.  By \eqref{pqclosedproposition:Lepieq}, this condition also implies $q_{\bar{u}}$ is closed. In addition, from the equivalence between \eqref{pqclosedproposition:eq1} and $\bar{u}\in \operatorname{rint} \operatorname{dom}(p_0)$ and Lemma~\ref{pqproperlemmalabel} we have that $q_{\bar{u}}$ is proper. Convexity of $q_{\bar{u}}$ follows from convexity of $f^*$ by \cite[Propositon 2.22]{rockafellar2009variational}.

The second sufficient condition follows from \cite[Proposition 10.21, Theorem 11.14 and  Proposition 11.32]{rockafellar2009variational}.
 \Halmos \end{proof}

\begin{proposition}\label{pqclosedproposition:primal:label}
Let $f\in \Gamma_0\left(\mathbb{R}^n\times \mathbb{R}^m\right)$, $\bar{v}\in \mathbb{R}^n$, $\bar{u}\in \mathbb{R}^m$, and $p_{\bar{v}}:\mathbb{R}^m\to \overline{\mathbb{R}}$, $q_{\bar{u}}:\mathbb{R}^n\to\overline{\mathbb{R}}$ be the functions given by \eqref{pqdef}.

Then $p_{\bar{v}} \in\Gamma_0\left(\mathbb{R}^m\right)$ if any of the following holds
\begin{itemize}
\item $\bar{v}\in \operatorname{rint} \operatorname{dom}(q_0)$, or
\item $\bar{v}\in  \operatorname{dom}(q_0)$ and $f$ is PWLQ.
\end{itemize}
\end{proposition}
\begin{proof}{\textbf{Proof}\hspace*{0.5em}}
 Analogous to the proof of   Proposition~\ref{pqclosedproposition:label}.
 \Halmos \end{proof}

\subsubsection{Conditions for optimal value attainment}

Non-emptyness of the sub-differentials in \eqref{newGNLP:args} can be ensured through standard results such as the following lemma.

\begin{restatable}{lemma}{subdifflemma}\label{subdifflemma:label}
Let $h:\mathbb{R}^n\to\overline{\mathbb{R}}$ be a convex function and $\bar{x}\in \mathbb{R}^n$. Then,  $\partial h(\bar{x})\neq \emptyset$ if any of the following holds
\begin{itemize}
    \item $\bar{x}\in \operatorname{rint} \operatorname{dom}(h)$ and $h(\bar{x})>-\infty$, or
    \item $\bar{x}\in \operatorname{dom}(h)$ and h is piecewise linear-quadratic.
\end{itemize}
\end{restatable}
\begin{proof}{\textbf{Proof}\hspace*{0.5em}}
For the first case, we begin by noting that if $\bar{x}\in \operatorname{rint} \operatorname{dom}(h)$ and $h(\bar{x})>-\infty$, then $h$ is proper. Indeed, $h(\bar{x})\in \mathbb{R}$ so $h\not\equiv +\infty$. In addition, if there exist $\bar{x}'\in \operatorname{dom}(h)$ such that $h(\bar{x}')=-\infty$, then by convexity of $h$ and $\bar{x}\in \operatorname{rint} \operatorname{dom}(h)$, there exist $\lambda>0$ such that \[h(\bar{x})\leq \lambda h\left(\frac{\bar{x}+(\lambda-1)\bar{x}'}{\lambda}\right)+(1-\lambda)h\left(\bar{x}'\right)=-\infty,\]
which contradicts $h(\bar{x})>-\infty$. The result then follows from \cite[Theorem 23.4]{rockafellar2015convex}.

The second case follows from \cite[Propositon 10.21]{rockafellar2009variational}.
 \Halmos \end{proof}

\subsubsection{Slater-like conditions}
\strongattainmentcoro*

\begin{proof}{\textbf{Proof of Corollary~\ref{strongattainmentcoro:label}}\hspace*{0.5em}}
\label{strongattainmentcoro:label:proof}

By \cite[Theorem 11.14 and Proposition 11.32]{rockafellar2009variational} we have that if $f$ is PWLQ then so are $p_{\bar{v}}$ and $q_{\bar{u}}$. The bullets follow from Proposition~\ref{pqclosedproposition:label}, Lemmas~\ref{pqproperlemmalabel} and \ref{subdifflemma:label}. Theorem~\ref{NewConvexStrongDual} (and $q_{\bar{u}}$ being proper) then implies $-q_{\bar{u}}\left(\bar{v}\right)=p_{\bar{v}}\left(\bar{u}\right)\in\mathbb{R}\cup \{-\infty\}$ with $-q_{\bar{u}}\left(\bar{v}\right)\in\mathbb{R}\cup \{-\infty\}$. If $p_{\bar{v}}\left(\bar{u}\right)>-\infty$, then 
\eqref{dual:attainment:eq} holds for any $\bar{y}\in\partial p_{\bar{v}}(\bar{u})$ by Theorem~\ref{NewConvexStrongDual}. If $q_{\bar{u}}\left(\bar{v}\right)=-p_{\bar{v}}\left(\bar{u}\right)=+\infty$, then \eqref{dual:attainment:eq} holds for any $\bar{y}\in\mathbb{R}^m$ by the definition of $q_{\bar{u}}$ in \eqref{qdef}.
\Halmos \end{proof}

\strongattainmentcorodual*
\begin{proof}{\textbf{Proof of Corollary~\ref{strongattainmentcorodual:label}}\hspace*{0.5em}}
\label{strongattainmentcorodual:label:proof}

 Analogous to the proof of   Corollary~\ref{strongattainmentcoro:label} except for using Proposition~\ref{pqclosedproposition:primal:label} instead of Proposition~\ref{pqclosedproposition:label}.
 \Halmos \end{proof}

Finally, it is worth noting that the more general strong duality conditions of $p_{\bar{v}}$ or $q_{\bar{u}}$ being proper and closed are not sufficient to ensure optimal value attainments. Examples of such attainment failures can be found in Section~\ref{nonattainmentexamples}.

\subsubsection{Optimal value non-attainment examples}\label{nonattainmentexamples}

\begin{example}
Let $f:\mathbb{R}\times \mathbb{R}\to\overline{\mathbb{R}}$, be give by
$f(x,u):=x+\delta_{\mathbb{R}_-}\left(x^2+u\right)$. We can check (e.g. using Lemma~\ref{gnlplagrangianlemma}) that
\begin{align*}
    f^*(v,y)=\begin{cases}\frac{1}{4}\frac{(v-1)^2}{y}&y> 0\\
0 &v=1,y=0\\
+\infty &\text{o.w.}\end{cases}, \quad
p_0(u)=\begin{cases}-\sqrt{-u} &u\leq 0\\+\infty &\text{o.w.}\end{cases}
\end{align*}
and $q_0(v)\equiv 0$. Hence, $p_0$ and $q_0$ are proper and closed, so by Theorem~\ref{NewConvexStrongDual} both  \eqref{newweakGNLPduality} and \eqref{newdualweakGNLPduality} yield 
\[0=\inf_{x\in \mathbb{R}}\left\{x\,:\, x^2\leq 0\right\}=\inf_{x\in \mathbb{R}}f(x,0)=-\inf_{y\in \mathbb{R}}f^*(0,y)=-\inf_{y\in \mathbb{R}}\left\{\frac{1}{4 y}\,:\, y\geq 0\right\}.\]
In addition,  $\partial q_{0}(0)=\{0\}\neq\emptyset$ so \eqref{newGNLP:args} implies the primal optimal value is attained: 
\[\{0\}=\arg \min_{x\in \mathbb{R}}\left\{x\,:\, x^2\leq 0\right\}.\]
However, even thought $p_0(0)\in \mathbb{R}$, we have $\partial p_0(0)=\emptyset$ so \eqref{newGNLP:args} implies the dual optimal value is not attained.

We can check that a crucial difference between the primal and dual constraint qualification conditions is that $0\in  \operatorname{rint} \operatorname{dom}(q_0)$, while $0\not\in  \operatorname{rint} \operatorname{dom}(p_0)$.
\end{example}

\begin{example}
Let $f:\mathbb{R}\times \mathbb{R}\to\overline{\mathbb{R}}$, be given by
$f(x,u):=\exp(x)+u$. We can check  that
\begin{align*}
    f^*(v,y)=\begin{cases}v \ln(v) -v &v> 0,y=1\\
0 &v=0,y=1\\
+\infty &\text{o.w.}\end{cases}
\end{align*}
$p_0(u)=u$ and $q_0(v)=f^*(v,1)$. Hence, $p_0$ and $q_0$ are proper and closed, so by Theorem~\ref{NewConvexStrongDual} both  \eqref{newweakGNLPduality} and \eqref{newdualweakGNLPduality} yield 
\[0=\inf_{x\in \mathbb{R}}\exp(x)=\inf_{x\in \mathbb{R}}f(x,0)=-\inf_{y\in \mathbb{R}}f^*(0,y)=-\inf_{y\in \mathbb{R}}\left\{0\,:\, y= 1\right\}.\]
In addition,  $\partial p_{0}(0)=\{1\}\neq\emptyset$ so \eqref{newGNLP:args} implies the dual optimal value is attained: 
\[\{1\}=\arg \min_{y\in \mathbb{R}}\left\{0\,:\, y=1\right\}.\]
However, even thought $q_0(0)\in \mathbb{R}$, we have $\partial q_0(0)=\emptyset$ so \eqref{newGNLP:args} implies the primal optimal value is not attained.

We can check that a crucial difference between the primal and dual constraint qualification conditions is that $0\in  \operatorname{rint} \operatorname{dom}(p_0)$, while $0\not\in  \operatorname{rint} \operatorname{dom}(q_0)$.
\end{example}

\end{APPENDICES}
%
%

\end{document}